\documentclass[reqno,twoside,11pt,english]{amsart}
\usepackage{amsmath,amsfonts,amssymb,amsthm,epsfig}
\newcommand{\RNum}[1]{\uppercase\expandafter{\romannumeral #1\relax}}
\usepackage{comment}

\usepackage{color}
\usepackage[final,colorlinks,linkcolor=blue,anchorcolor=red,citecolor=blue]{hyperref}
\usepackage{graphicx}
\usepackage{titletoc,epsf}
\usepackage{amsmath,amsfonts,latexsym,amsthm,amsxtra,amssymb,bbm}
\allowdisplaybreaks
\usepackage{graphicx,float}
\usepackage{tikz}
\usetikzlibrary{intersections,patterns}
\usetikzlibrary{calc,spy}
\usepackage{ifthen}
\usepackage{float}

\allowdisplaybreaks

\newtheorem{thm}{Theorem}[section]
\newtheorem{lem}[thm]{Lemma}
\newtheorem{op}[thm]{Open Problem}
\newtheorem{cor}[thm]{Corollary}
\newtheorem{prop}[thm]{Proposition}
\newtheorem{assertion}{Assertion}[section]

\newtheorem{step}{Step}[section]

\newtheorem{cl}{Claim}[section]

\newtheorem{ca}{Case}
\newtheorem{sca}[section]{Subcase}
\newtheorem{scl}[section]{Subclaim}
\newtheorem{conj}[equation]{Conjecture}

\theoremstyle{definition}
\newtheorem{defn}[thm]{Definition}

\newtheorem{ques}[equation]{Question}
\newtheorem{rem}[thm]{Remark}
\newtheorem{exam}[thm]{Example}

\newcounter {own}
\def\theown {\thesection       .\arabic{own}}
\numberwithin{equation}{section}

\newenvironment{pf}[1][]{%
	\vskip 3mm
	\noindent
	\ifthenelse{\equal{#1}{}}%
	{{\slshape Proof. }}%
	{{\slshape #1.} }%
}%
{\qed\bigskip}

\newtheorem{Thm}{Theorem}


\newcommand{\IR}{{\mathbb R}}

\newcommand{\dist}{{\operatorname{dist}}}

\newcommand{\bas}{\begin{assertion}}
	\newcommand{\eas}{\end{assertion}}
\newcommand{\ben}{\begin{enumerate}}
	\newcommand{\een}{\end{enumerate}}
\newcommand{\bst}{\begin{step}}
	\newcommand{\est}{\end{step}}

\def\be{\begin{equation}}
	\def\ee{\end{equation}}

\newcommand{\bee}{\begin{enumerate}}
	\newcommand{\eee}{\end{enumerate}}

\newcommand{\blem}{\begin{lem}}
	\newcommand{\elem}{\end{lem}}
\newcommand{\bthm}{\begin{thm}}
	\newcommand{\ethm}{\end{thm}}
\newcommand{\bcor}{\begin{cor}}
	\newcommand{\ecor}{\end{cor}}
\newcommand{\beg}{\begin{exam}}
	\newcommand{\eeg}{\end{exam}}
\newcommand{\begs}{\begin{examples}}
	\newcommand{\eegs}{\end{examples}}
\newcommand{\bdefe}{\begin{defn}}
	\newcommand{\edefe}{\end{defn}}
\newcommand{\bprob}{\begin{prob}}
	\newcommand{\eprob}{\end{prob}}
\newcommand{\bques}{\begin{ques}}
	\newcommand{\eques}{\end{ques}}
\newcommand{\bei}{\begin{itemize}}
	\newcommand{\eei}{\end{itemize}}
\newcommand{\bcon}{\begin{conj}}
	\newcommand{\econ}{\end{conj}}
\newcommand{\bop}{\begin{op}}
	\newcommand{\eop}{\end{op}}

\newcommand{\bstep}{\begin{step}}
	\newcommand{\estep}{\end{step}}

\newcommand{\bca}{\begin{ca}}
	\newcommand{\eca}{\end{ca}}
\newcommand{\bsca}{\begin{sca}}
	\newcommand{\esca}{\end{sca}}

\newcommand{\bcl}{\begin{cl}}
	\newcommand{\ecl}{\end{cl}}

\newcommand{\bscl}{\begin{scl}}
	\newcommand{\escl}{\end{scl}}

\newcommand{\bcons}{\begin{conjs}}
	\newcommand{\econs}{\end{conjs}}
\newcommand{\bprop}{\begin{prop}}
	\newcommand{\eprop}{\end{prop}}
\newcommand{\br}{\begin{rem}}
	\newcommand{\er}{\end{rem}}
\newcommand{\brs}{\begin{rems}}
	\newcommand{\ers}{\end{rems}}
\newcommand{\bo}{\begin{obser}}
	\newcommand{\eo}{\end{obser}}
\newcommand{\bos}{\begin{obsers}}
	\newcommand{\eos}{\end{obsers}}
\newcommand{\bpf}{\begin{pf}}
	\newcommand{\epf}{\end{pf}}
\newcommand{\ba}{\begin{array}}
	\newcommand{\ea}{\end{array}}
\newcommand{\beq}{\begin{eqnarray}}
	\newcommand{\beqq}{\begin{eqnarray*}}
		\newcommand{\eeq}{\end{eqnarray}}
	\newcommand{\eeqq}{\end{eqnarray*}}

\newcommand{\ds}{\displaystyle}

\newcounter{minutes}
\divide\time by 60
\newcounter{hours}
\multiply\time by 60 \addtocounter{minutes}{-\time}

\begin{document}

\title[{\tiny  Uniformizing non-proper Gromov Hyperbolic Spaces}]{Uniformizing  non-proper Gromov Hyperbolic Spaces}

\author[S.-J. Gao, C.-Y. Guo, M. He ]{Shu-Jing Gao, Chang-Yu Guo and Miao He}

\address[S.-J. Gao]{School of Mathematics, Shandong University, Jinan and Research Center for Mathematics and Interdisciplinary Sciences, Shandong University, Qingdao, P. R. China}
\email{gsjing9@163.com}

\address[C.-Y. Guo]{Research Center for Mathematics and Interdisciplinary Sciences, Shandong University, 266237, Qingdao, P. R. China, and Department of Physics and Mathematics, University of Eastern Finland, 80101, Joensuu, Finland}
\email{changyu.guo@sdu.edu.cn}

\address[M. He]{Research Center for Mathematics and Interdisciplinary Sciences, Shandong University, Qingdao, P. R. China and Frontiers Science Center for Nonlinear Expectations, Ministry of Education, P. R. China}
\email{hm293428@163.com}



\date{}

\begin{abstract}
In this paper, we extend a large part of the uniformization theory of Bonk-Heinonen-Koskela [Asterisque 2001] to length spaces that are not necessarily proper or geodesic. Among other things, we show that there is a one-to-one correspondence between the quasiisometry classes of complete roughly starlike Gromov hyperbolic spaces and the quasisimilarity classes of bounded uniform spaces, which provides an affirmative solution to an open question of Bonk-Heinonen-Koskela. Our approach relies crucially on the work of V\"ais\"al\"a [Expo. Math. 2005], who investigated in depth Gromov hyperbolic spaces that are not necessarily proper or geodesic. One key new ingredient is to use the so-called (quasihyperbolic) $(c,\mu)$-quasigeodesic as a suitable substitute for quasihyperbolic geodesic.
\end{abstract}


\thanks{All authors are supported by the Young Scientist Program of the Ministry of Science and Technology of China (No.~2021YFA1002200), the National Science Foundation of China (No.~12671095) and the Taishan Scholar Project.  }

\maketitle
\tableofcontents

\section{Introduction}\label{sec-1}
\subsection{Background}

A metric space $X=(X,d)$ is called a \emph{length space}, if for all $x,y \in X$, $d(x,y)= \inf\{\ell_d(\alpha) : \alpha: x \curvearrowright y\}$. Throughout this article, all metric spaces are assumed to be length. The following important concept was introduced by Gromov \cite{Gromov1987}. 
\begin{defn}[Gromov hyperbolic spaces]\label{def:Gromov hyperbolic}
	A length space $X$ is called \emph{$\delta$-Gromov hyperbolic} if for all $x,y,z,p \in X$, it holds 
	\[
	(x|z)_p \geq (x|y)_p \wedge (y|z)_p - \delta, 
	\]
	where the \emph{Gromov product} $(x|y)_p = \frac{1}{2}(|x-p| + |y-p| - |x-y|)$.
\end{defn}

Gromov hyperbolic spaces constitute one of the most important classes of metric spaces in modern geometry. These spaces provide a far-reaching generalization of classical hyperbolic geometry and have found profound applications in geometric group theory, complex dynamics, analysis on metric spaces, and geometric topology; see for instance \cite{BridsonHaefliger1999,Vaisala}. Since bounded spaces are always Gromov hyperbolic, we assume that all Gromov hyperbolic spaces considered in this paper are unbounded. 

In connection with quasiconformal analysis and geometry, a fundamental problem in the theory of Gromov hyperbolic spaces is the following:
\smallskip

\textbf{Uniformization Problem}: given a Gromov hyperbolic space $X$, construct a canonical metric deformation that transforms $X$ into a model space with enhanced regularity/geometric properties. 
\smallskip 

In quasiconformal analysis, an ideal candidate for such a model space is the class of \emph{uniform spaces}, introduced by Martio and Sarvas \cite{MS}, based on earlier ideas of John \cite{Jo}.  

\begin{defn}[Uniform spaces]\label{def: uniform domain}
	A noncomplete rectifiably connected metric space $X=(X,d)$ is called  {\it $A$-uniform}, if there is a constant $A\geq 1$ and each pair of points $z_{1},z_{2}$ in $X$ can be joined by a rectifiable curve $\gamma$ in $X$ satisfying
	\begin{enumerate}
		\item\label{con1} \textbf{Double cone condition}: $\ds\min_{j=1,2}\{\ell_d (\gamma [z_j, z])\}\leq A\, d_{X}(z)$ for all $z\in \gamma$;
		\item\label{con2}\textbf{Quasiconvexity condition}: $\ell_d(\gamma)\leq A\,d(z_{1},z_{2})$,
	\end{enumerate}
	where $\ell_d(\gamma)$ denotes the arc-length of $\gamma$ with respect to the metric $d$,
	$\gamma[z_{j},z]$ the subcurve of $\gamma$ between $z_{j}$ and $z$, and $d_X(z):=d(z,\partial X)$ denotes the distance from $z$ to the boundary of $X$. In an $A$-uniform space $X$, any curve $\gamma\subset X$, which satisfies conditions \eqref{con1} and \eqref{con2} above, is called an {\it $A$-uniform curve}.
\end{defn}


In Definition \ref{def: uniform domain} (\ref{con2}), if $d(z_{1},z_{2})$ is replaced by the inner distance $\sigma_{X}(z_{1},z_{2})$ defined by 
$$\sigma_{X}(x,y)=\inf\{\ell_d(\gamma) : \gamma: x \curvearrowright y\; \text{in} \; X\},$$
 then we call the metric space $X=(X,d)$ {\it $A$-inner uniform} and the corresponding curve an {\it $A$-inner uniform curve}.
 
Both Gromov hyperbolic spaces and (inner) uniform spaces are central objects of study in the theory of quasiconformal maps and analysis on metric spaces; see for instance \cite{BB4,BHK,BS,GGGKN-2024,Guo-2015,Guo-Huang-Wang-2025-III,HRWZ-2025,Hajlasz-Koskela-2000,Hei-Book-2001,HeKo,Jo81,HLPW,Klm2014,MS,Martin-1985,Yang-2026-RM,zhou-2026-bsm}.

The preceding uniformization problem was studied in depth by Bonk, Heinonen and Koskela in their seminal work \cite{BHK}. In their main result, they established the following deep theorem. 
\begin{Thm}[{\cite[Theorem 1.1]{BHK}}]\label{ThmA:Bonk-Heinonen-Koskela}
	There is a one-to-one correspondence between the quasiisometry classes of \emph{proper geodesic} roughly starlike Gromov hyperbolic spaces and the quasisimilarity classes of bounded \emph{locally compact} uniform spaces.
\end{Thm}

To interpret Theorem \ref{ThmA:Bonk-Heinonen-Koskela} accurately, we need to recall some elementary definitions. A homeomorphism $f \colon (X, d) \to (X', d')$ between two metric spaces is a \textit{quasiisometry}, or an \textit{$L$-quasiisometry}, if $L \geqslant 1$ and
\begin{equation}\label{eq:1.16}
	\frac{1}{L} \, d(x,y) \leqslant d'(f(x), f(y)) \leqslant L \, d(x,y)
\end{equation}
for all $x, y \in X$.
Thus quasiisometries are the same as \textit{bi-Lipschitz homeomorphisms}.

A homeomorphism $f \colon (X, d) \to (X', d')$ between two metric spaces is \textit{quasisymmetric}, or \textit{$\eta$-quasisymmetric}, if $\eta \colon [0, \infty) \to [0, \infty)$ is a homeomorphism such that
\begin{equation}\label{eq:quasisymmetric}
	\frac{d'(f(x), f(y))}{d'(f(x), f(z))} \leqslant \eta\left(\frac{d(x,y)}{d(x,z)}\right)
\end{equation}
for all triples of distinct points $x, y, z$ in $X$. 

\begin{defn}[Quasisimilarity]\label{def:quasisimilarity}
	A homeomorphism $f \colon (\Omega, d) \to (\Omega', d')$ between two noncomplete metric spaces is a \textit{quasisimilarity}, with \textit{data} $(\eta, L, \lambda)$, where $L \geqslant 1$ and $0 < \lambda \leqslant 1$, if	$f$ is $\eta$-quasisymmetric	and if for each $x \in \Omega$ there is $c_x > 0$ such that
	\begin{equation}\label{eq:con2}
		\frac{c_x}{L} \, d(z,y) \leqslant d'(f(z), f(y)) \leqslant L c_x \, d(z,y)
	\end{equation}
	whenever $z, y \in B(x,\lambda d_\Omega (x))$.
\end{defn}

There are two main operations towards the uniformization theory of Bonk-Heinonen-Koskela \cite{BHK}. One operation is \emph{uniformization} $\mathcal{D}\colon X \longrightarrow X_\varepsilon$, the letter $\mathcal{D}$ abbreviating \emph{dampening}, and the other one is
\emph{quasihyperbolization} $\mathcal{Q}\colon (\Omega,d) \longrightarrow (\Omega, k)$.  






To be precise, let $X=(X,|\cdot|)$ be a metric space with a fixed base point $p \in X$. For each $\varepsilon > 0$, consider the conformal deformation induced by the density
\begin{equation}\label{eq:density}
	\rho_{\varepsilon}(x) = \exp\bigl(-\varepsilon|x-p|\bigr), \quad x \in X,
\end{equation}
and define the deformed metric by
\begin{equation}\label{eq:conformal-metric}
	d_\varepsilon(a,b) = \inf_{\gamma} \int_{\gamma} \rho_{\varepsilon} \, ds,
\end{equation}
where the infimum is taken over all rectifiable curves in $(X,|\cdot|)$ joining $a$ and $b$. The resulting metric space is denoted $X_{\varepsilon} = (X, d_ \varepsilon)$. Furthermore, for $x\in X_\varepsilon$, define $$d_\varepsilon(x) = d_\varepsilon(x,  \partial X_\varepsilon),$$
where $ \partial X_\varepsilon = \overline{X}_\varepsilon \setminus X_\varepsilon$.
For any rectifiable curve $\gamma$ in $X$ we have
\begin{equation}\label{eq:length}
	\ell_\varepsilon(\gamma) = \int_\gamma \rho_\varepsilon \, ds,
\end{equation}
provided the identity map $(X, |x-y|) \to (X, \ell)$ is a homeomorphism. 

Next, we recall the quasihyperbolic metric induced by a metric $d$. This concept was initially introduced by Gehring and Palka \cite{GePa} for domains in $\IR^n$, and then, has been extensively studied in \cite{Geo}.
The {\it quasihyperbolic length} of a rectifiable curve
$\gamma$
in a noncomplete metric space $(\Omega,d)$ is defined as
$$
\ell_{k}(\gamma):=\int_{\gamma}\frac{|dz|}{d_\Omega(z)}.
$$
For any $z_1$, $z_2$ in $\Omega$, the {\it quasihyperbolic distance}
$k(z_1,z_2)$ between $z_1$ and $z_2$ is set to be
$$k(z_1,z_2)=\inf_{\gamma}\{\ell_k(\gamma)\},
$$
where the infimum is taken over all rectifiable curves $\gamma$
joining $z_1$ and $z_2$ in $\Omega$. A rectifiable curve $\gamma$ from $z_1$ to $z_2$ is called a {\it quasihyperbolic geodesic} if
$\ell_k(\gamma)=k(z_1,z_2)$. Clearly, each subcurve of a quasihyperbolic
geodesic is a quasihyperbolic geodesic.

Bonk-Heinonen-Koskela proved in \cite[Proposition 4.15]{BHK} that if $X$ is proper geodesic $\delta$-Gromov hyperbolic space, then quasihyperbolic geodesics in the conformal deformations $X_\varepsilon=(X,d_{\varepsilon})$ are $A(\delta)$-uniform curves for $0<\varepsilon<\varepsilon_0(\delta)$. In particular, all conformal deformations $X_\varepsilon$ are bounded $A(\delta)$-uniform spaces. In \cite[Proposition 4.28]{BHK}, they further obtained that if $(\Omega,d)$ is $A$-uniform, then the identity map id$\colon (\Omega,d)\to \Omega_{\varepsilon}$ is quasisimilar with data depending only on $A$ and $\varepsilon$\footnote{In \cite[Proposition 4.28]{BHK}, it was inaccurately stated that the data associated with  id depend only on $A$. In Example \ref{exam:counter-example} below, we give a counter-example to illustrate the additional dependence on $\varepsilon$} for all $0<\varepsilon\leq \varepsilon_0(A)$. Conversely, given a bounded uniform space $\Omega=(\Omega,d)$, let $k$ be the associated quasihyperbolic metric. Then they proved in \cite[Theorem 3.6]{BHK} that the resulting metric space $(\Omega,k)$ is proper geodesic and roughly starlike Gromov hyperbolic space. 

The one-to-one correspondence between two types of isomorphism classes of metric spaces, as stated in Theorem \ref{ThmA:Bonk-Heinonen-Koskela}, is given in terms of the operations $\mathcal{D}$ and $\mathcal{Q}$. More precisely, the composition $\mathcal{Q}\circ \mathcal{D}$ takes a proper geodesic and roughly starlike $\delta$-Gromov hyperbolic space back to its quasiisometry class, and the composition $\mathcal{D}\circ \mathcal{Q}$ preserves the quasisimilarity type of bounded locally compact uniform spaces.

Motivated by the theory of quasiconformal maps in infinite-dimensional Banach spaces, Bonk-Heinonen-Koskela asked the following  open question in \cite[Last paragraph on page 5]{BHK}.
\smallskip

\textbf{Question (Bonk-Heinonen-Koskela)}. It would be interesting to know \emph{if there is a general relationship between uniformity and Gromov hyperbolicity that could cover infinite-dimensional situations as well}.
\smallskip 

Motivated by the above open question, we aim at extending Theorem \ref{ThmA:Bonk-Heinonen-Koskela} to more general non-proper metric spaces. Note that as was pointed out by the authors in \cite[Last paragraph on page 5]{BHK}, their method largely will not work due to lack of compactness. On the other hand, V\"ais\"al\"a \cite{Vaisala} has already developed a rich theory of Gromov hyperbolic spaces that are not necessarily proper or geodesic. 

\subsection{Main results}
Our main result of this paper shows that Theorem \ref{ThmA:Bonk-Heinonen-Koskela} remains valid in the general setting of complete roughly starlike Gromov hyperbolic spaces and the associated bounded uniform spaces. In particular, it includes many infinite dimensional spaces such as Banach spaces and Alexandrov spaces (such as CAT(0) spaces). 
\begin{thm}\label{thm:one-to-one correspondence}
	There is a one-to-one correspondence between the quasiisometry classes of complete roughly starlike Gromov hyperbolic spaces and the quasisimilarity classes of bounded uniform spaces.
\end{thm}
  
For the proof of Theorem \ref{thm:one-to-one correspondence}, we shall follow the strategy of Bonk-Heinonen-Koskela \cite{BHK}, making use of the uniformization operator $\mathcal{D}\colon X \longrightarrow X_\varepsilon$ and the quasihyperbolization operator $\mathcal{Q}\colon (\Omega,d) \longrightarrow (\Omega, k)$. As an extension of \cite[Proposition 4.15]{BHK} mentioned above, we shall prove in Theorem \ref{thm:uniform} below that $\mathcal{D}$ sends complete $\delta$-Gromov hyperbolic spaces to bounded $A(\delta)$-uniform spaces for $0<\varepsilon<\varepsilon_0(\delta)$. Unlike in the setting of Bonk-Heinonen-Koskela, for non-proper spaces, quasihyperbolic geodesics are not always available. Motivated by the class of quasigeodesics introduced by V\"ais\"al\"a (see e.g.~\cite{Vaisala1999}) and by the recent work on dimension-free Gehring-Hayman inequalities \cite{Guo-Huang-Wang-2025-I}, we shall study a variant curve family, consisting of \emph{(quasihyperbolic) quasigeodesics}, as a suitable substitute for quasihyperbolic geodesics.

Let $c\geq1$ and $\mu\geq0$, an arc $\gamma: a \curvearrowright b$ in a metric space $(X,d)$ is called \emph{$(c,\mu)$-rough quasigeodesic} if for all $z, w \in \gamma$,
\begin{equation*}\label{eq:quasigeodesic}
	\ell_d(\gamma[z,w]) \leq c|z-w|+\mu.
\end{equation*}
When $\mu=0$, $\gamma$ is called \emph{$c$-quasigeodesic}, and if in addition $c=1$, $\gamma$ is a \emph{geodesic}.
\begin{defn}[$(c,\mu)$-quasigeodesic]\label{def:rough quasigeodesic}
For $c\geq 1$ and $\mu\geq 0$, we set 
\beqq
\begin{aligned}
	\Lambda_{x,y}^{c,\mu}(X, d) = \Bigl\{ \gamma \subset  X :\, &\gamma: x \curvearrowright y
	\text{ is } (c,\mu_0)\text{-rough quasigeodesic with }  \mu_0 = \mu \min\{|x-y|, 1\} \Bigr\}.
\end{aligned}
\eeqq
Each $\gamma \in \Lambda_{x,y}^{c,\mu}(X, d)$ is called a \emph{$(c,\mu)$-quasigeodesic} in $(X,d)$. 

If the metric $d$ is replaced by the associated quasihyperbolic metric $k$, we call 
$\gamma \in \Lambda_{x,y}^{c,\mu}(X, k)$ a \emph{quasihyperbolic $(c,\mu)$-quasigeodesic} in $X$. 
\end{defn}

If $X=(X,d)$ is a length space, then for each $c\geq 1$ and $\mu>0$, it is easy to see that for any pair of points $x,y\in X$, $\Lambda_{x,y}^{c,\mu}(X, d)\neq \emptyset.$

For each $(c,\mu)$-quasigeodesic $\gamma \in \Lambda_{x,y}^{c,\mu}(X, d)$, $z,w \in \gamma$, we have 
\beq\label{eq:prop of subcurve}
\ell_d(\gamma[z,w]) \leq c\,|z-w| + \mu_0\leq c|z-w| +\mu.
\eeq
This implies that $\gamma$ is $(c,\mu)$-rough quasigeodesic. Furthermore, we also have 
\beq\label{eq:prop of curve}
\ell_d(\gamma) \leq c\,|x-y| + \mu\,|x-y|	= (c+\mu)\,|x-y|.
\eeq

As an extension of \cite[Proposition 4.15]{BHK}, we show that $(c,\mu)$-quasigeodesics in complete Gromov hyperbolic spaces are uniform curves in their conformal deformations, quantitatively.
\begin{thm}\label{thm:uniform}
	Let $X=(X,d)$ be a complete $\delta$-Gromov hyperbolic space, $c\geq 1$ and $\mu> 0$. Then there exist constants $\varepsilon_0=\varepsilon_0(\delta,c,\mu)>0$ and $A=A(\delta,c,\mu)\geq 1$ such that for all $0 < \varepsilon \leq \varepsilon_0$, every $(c,\mu)$-quasigeodesic $\gamma$ in $(X,d)$ is an $A$-uniform curve in the conformally deformed space $X_\varepsilon = (X,d_\varepsilon)$. In particular, 
	all $X_\varepsilon$ are bounded $A_0$-uniform spaces for $0 < \varepsilon \leq \varepsilon_1$, where $A_0=A_0(\delta)=A(\delta,1,1)$ and $\varepsilon_1=\varepsilon_1(\delta)=\varepsilon_0(\delta,1,1)$. 
\end{thm}


The proof of Theorem \ref{thm:uniform} relies crucially on the following Gehring-Hayman inequality for $(c,\mu)$-quasigeodesics in Gromov hyperbolic spaces.

\begin{thm}[Gehring-Hayman Inequality]\label{thm:G-H}
	Let $X=(X,|\cdot|)$ be a $\delta$-Gromov hyperbolic space, $c\geq 1$ and $\mu> 0$. Let $\rho: X \to (0,\infty)$ be a continuous density satisfying, for some $\varepsilon > 0$ and $C \geq 1$,
	\begin{equation}\label{eq:Harnack-general}
		\frac{1}{C} \exp\left(-\varepsilon|x-y|\right) \leq \frac{\rho(x)}{\rho(y)} \leq C \exp\left(\varepsilon|x-y|\right).
	\end{equation}
	Then there exist constants $\varepsilon_0 = \varepsilon_0(\delta, C, c, \mu) > 0$ and $K = K(\delta, C, c,\mu) \geq 1$ such that if $0 < \varepsilon \leq \varepsilon_0$, then for each $\gamma \in \Lambda_{x,y}^{c,\mu}(X,d)$ and each rectifiable curve $\alpha: x \curvearrowright y$, it holds 
	\begin{equation}\label{eq:G-H}
		\ell_\rho(\gamma) \leq K \ell_\rho(\alpha).
	\end{equation}
\end{thm}

The Gehring-Hayman inequality as above will be one of our central tools in the proofs of our main results in Sections \ref{sec:uniformity of quasigeodesic}, \ref{sec:uniform spaces are Gromov hyperbolic} and \ref{sec:final uniformization}. 
Recall that the metric space $X_{\varepsilon} = (X, d_ \varepsilon)$ is defined via \eqref{eq:conformal-metric} with the conformal density given by \eqref{eq:density}.
The triangle inequality implies that
\begin{equation}\label{eq:inequality}
	\exp\left(-\varepsilon|x-y|\right)\leq \frac{\rho_{\varepsilon}(x)}{\rho_{\varepsilon}(y)} \leq \exp\left(\varepsilon|x-y|\right).
\end{equation}
This implies in particular that the conformal density as in \eqref{eq:density} satisfies \eqref{eq:Harnack-general} with constant $C=1$. 


Bonk-Heinonen-Koskela \cite[Theorem 3.6]{BHK} proved that every locally compact uniform space, when equipped with the quasihyperbolic metric, is Gromov hyperbolic, quantitatively. Moreover, if the uniform space is additionally bounded, then it is roughly starlike. As a by-product of our general method, we may remove the local compactness assumption.  
 
\begin{thm}\label{thm:Gromov and roughly starlike}
	Let $(\Omega,d)$ be an $A$-uniform space 
	and $k$ the associated quasihyperbolic metric. Then $(\Omega,k)$ is a complete $\delta_1$-Gromov hyperbolic space with $\delta_1=\delta_1(A).$ If additionally $(\Omega,d)$ is bounded, then $(\Omega,k)$ is roughly starlike. Furthermore, every $A$-inner uniform space is $\delta_1$-Gromov hyperbolic with $\delta_1=\delta_1(A)$.
\end{thm}

The final assertion in Theorem \ref{thm:Gromov and roughly starlike} can be viewed as an extension of \cite[Theorem 1.11]{BHK} to non-proper length spaces, where the authors proved inner uniform domains in $\mathbb{R}^n$, equipped with the quasihyperbolic metric, are Gromov hyperbolic. 

\subsection{Short comments on our approach}
We now briefly comment on the main differences with the preceding foundational work of Bonk-Heinonen-Koskela \cite{BHK}. As was pointed out there, their arguments largely rely on local compactness or properness and thus does not work for non-proper spaces. One key reason for the restriction of their method is that the analysis in \cite{BHK}  relies crucially on \emph{quasihyperbolic geodesics}, which does not always exist in non-proper or infinite-dimensional spaces. 

In the interesting work \cite{Vaisala}, V\"ais\"al\"a has successfully extended some results of Bonk-Heinonen-Koskela to Gromov hyperbolic spaces that are not necessarily proper or geodesic. The key idea of  V\"ais\"al\"a is to use \emph{short curves} (see Section \ref{sec:preliminaries} below for precise definition) to replace quasihyperbolic geodesics. Unlike quasihyperbolic geodesics, short curves exist in general length spaces and enjoy good quasihyperbolic estimates akin to quasihyperbolic geodesics. Using sequences of short curves, one can also define ``geodesic rays to a boundary point", which will be called \emph{roads}. This concept turns out to be useful for studying the identification of the Gromov boundary with the metric boundary of complete Gromov hyperbolic spaces. 

Our starting point is to \emph{find the most general classes of curves} so that the Gehring-Hayman inequality remains valid. In the second author's recent joint work \cite{Guo-Huang-Wang-2025-I}, we proved that the dimension-free Gehring-Hayman inequality holds for \emph{quasigeodesics} in general length spaces. This motivates us to consider the class of (quasihyperbolic) $(c,\mu)$-quasigeodesics as in Definition \ref{def:rough quasigeodesic}. Combining ideas from Bonk-Heinonen-Koskela \cite{BHK} and V\"ais\"al\"a \cite{Vaisala1999}, we are able to prove the Gehring-Hayman inequality for $(c,\mu)$-quasigeodesics, which is essentially the most general class of curves so that Gehring-Hayman inequality holds (as one can easily construct counter-examples to show that the Gehring-Hayman inequality fails for $(c,\mu)$-rough quasigeodesics). Thus, to a large extent, we follow a similar strategy as that of Bonk-Heinonen-Koskela \cite{BHK}, with technical difficulties arising from the generality of (quasihyperbolic) quasigeodesics. A good point is that the class of (quasihyperbolic) $(c,\mu)$-quasigeodesics, as considered in this paper, seems to be the correct family of curves to study the uniformization theory of Gromov hyperbolic spaces. 

\medskip 
\textbf{Organization of the paper.} The paper is organized as follows. Section \ref{sec:preliminaries} contains preliminaries on $h$-short arcs, Rips spaces, and some lemmas needed in the proofs of the main theorems. In Section \ref{sec:Gehring-Hayman}, we prove Theorem~\ref{thm:G-H}, the Gehring–Hayman inequality for 
$(c,\mu)$-quasigeodesics. In Section \ref{sec:uniformity of quasigeodesic}, we prove Theorem~\ref{thm:uniform}, which shows that the conformal deformation of a complete $\delta$-Gromov hyperbolic space is a bounded uniform space. In Section \ref{sec:uniform spaces are Gromov hyperbolic}, we prove Theorem~\ref{thm:Gromov and roughly starlike}. In the final section, Section \ref{sec:final uniformization}, we provide detailed proofs of Theorem~\ref{thm:one-to-one correspondence}. In the appendix, we include a sketch of the proof of Proposition \ref{prop:boundary-id}. 
\medskip 

\textbf{Notation.} For a metric space $(X,d)$, we often write $|x-y|$ for $d(x,y)$.  For $x \in X$ and $r > 0$, $B_d(x, r)$ denotes the open ball. The Hausdorff distance between sets $A,B \subset X$ is denoted $d_H(A,B)$. The notation $x\asymp_{c}y$ means that there exists a constant $C=C(c)>0$ such that $\frac{1}{C}y\leq x \leq C y$. Given two non-negative real numbers $a$ and $b$, we set
\[
a\wedge b=\min\{a,b\}\quad \text{and}\quad a\vee b=\max\{a,b\}.
\]


\section{Preliminaries}\label{sec:preliminaries}


\subsection{Some elementary concepts and estimates}
For $x,y\in \Omega$, we have the following two elementary estimates for quasihyperbolic metric:
\beq\label{eq:quasihyperbolic inequality}
k(x,y)\geq\log\left(1+\frac{\sigma_{\Omega}(x,y)}{d_\Omega(x)\wedge d_\Omega(y)}\right)\geq\log\left(1+\frac{d(x,y)}{d_\Omega(x)\wedge d_\Omega(y)}\right)\geq\left|\log\frac{d_{\Omega}(x)}{d_{\Omega}(y)}\right|
\eeq
and
\beq\label{eq:quasihyperbolic length inequality}
\ell_k(\gamma)\geq\log\left(1+
\frac{\ell_d(\gamma)}
{d_\Omega(x)\wedge d_\Omega(y)}\right).
\eeq

We recall the following lemma from \cite[Lemma 2.13]{BHK}, which gives a useful estimate for the quasihyperbolic metric in uniform/inner uniform domains.
\begin{lem}[{\cite[Lemma 2.13]{BHK}}]\label{lem:intrinsic-213}
	Let $\gamma:[0,1]\to\Omega$ be a rectifiable curve with endpoints
	$x=\gamma(0)$ and $y=\gamma(1)$. Suppose that, for some $A\ge1$ and for all $t\in[0,1]$, it holds
	\begin{equation*}
		\ell_d(\gamma[0,t])
		\wedge
		\ell_d(\gamma[t,1])
		\le
		Ad_\Omega(\gamma(t)).
	\end{equation*}
	Then
	\begin{equation}\label{eq:213-first}
		\begin{aligned}
			\ell_k(\gamma)
			&\le
			2A\log\left[
			\left(1+\frac{\ell_d(\gamma)}{d_\Omega(x)}\right)
			\left(1+\frac{\ell_d(\gamma)}{d_\Omega(y)}\right)
			\right]                                                   \\
			&\le
			4A\log\left(
			1+
			\frac{\ell_d(\gamma)}
			{d_\Omega(x)\wedge d_\Omega(y)}
			\right).
		\end{aligned}
	\end{equation}
	Consequently, if $(\Omega,d)$ is $A$-uniform, then for all $x,y\in\Omega$,
	\begin{equation}\label{eq:j-k-upper}
		k(x,y)
		\le
		4A^2\log\left(1+\frac{d(x,y)}{d_\Omega(x)\wedge d_\Omega(y)}\right).
	\end{equation}
	If $(\Omega,d)$ is an $A$-inner uniform, then for all $x,y\in\Omega$,
	\begin{equation}\label{eq:inner upper}
		k(x,y)
		\le
		4A^2\log\left(1+\frac{\sigma_\Omega(x,y)}{d_\Omega(x)\wedge d_\Omega(y)}\right).
	\end{equation}
\end{lem}

Next, we recall the following useful concept of short curves. 
\begin{defn}
	For $h\geq 0$, an arc $\alpha: x \curvearrowright y$ is \emph{$h$-short} if $\ell_d(\alpha) \leq |x-y| + h$.
\end{defn}
Note that in a length space X, for every pair of points $x,y\in X$ and every $h>0$, there exists an \emph{$h$-short} arc $\alpha: x \curvearrowright y$. Moreover, every subarc of an \emph{$h$-short} arc is an \emph{$h$-short} arc.

We follow the notations used by V\"{a}is\"{a}l\"{a}~\cite{Vaisala}. 
\begin{defn}[Induced subdivision]\label{def:induced subdivision}
Let $\Delta=(\alpha,\beta,\gamma)$ be an $h$-short triangle in $(X,d)$, which means $\alpha,\beta,\gamma$ are $h$-short, with vertices $x,y,z$, where $\gamma:x\curvearrowright y$. Let $x_\gamma,\,y_\gamma\in\gamma$ be the points determined by
\[
\ell_d(\gamma[x,x_\gamma])=(y|z)_x,
\quad
\ell_d(\gamma[y_\gamma,y])=(x|z)_y.
\]
Set $\gamma_x=\gamma[x,x_\gamma]$, $\gamma_y=\gamma[y_\gamma,y]$, and $\gamma^*=\gamma[x_\gamma,y_\gamma]$. Then $\gamma=\gamma_x\cup\gamma^*\cup\gamma_y$ is the subdivision of $\gamma$ induced by the third vertex $z$.
\end{defn}

In the theory of proper geodesic hyperbolic spaces, \textit{geodesic rays} have turned out to be useful. For example, one can define a Gromov boundary point as an equivalence class of geodesic rays. However in a general length space, geodesic rays may not exist. Following \cite{Vaisala}, we instead work with \textit{roads} -- sequences of $h$-short arcs that serve as their substitute.
\begin{defn}[Road]\label{def:road}
	Let $(X,d)$ be a metric space and let $\nu \geq 0$, $h \geq 0$. A \textit{$(\nu, h)$-road} in $X$ is a sequence $\bar{\alpha}$ of arcs $\alpha_i : x_i \curvearrowright y_i$ with the following properties:
	\begin{enumerate}
		\item\label{road1} Each $\alpha_i$ is $h$-short.
		\item\label{road2} The sequence of lengths $\ell_d(\alpha_i)$ is increasing and tends to $\infty$.
		\item\label{road3} For $i \leqslant j$, the length map $g_{ij} : \alpha_i \to \alpha_j$ with $g_{ij} x_i = x_j$ satisfies $|g_{ij} u - u| \leqslant \nu$ for all $u \in \alpha_i$.
	\end{enumerate}
\end{defn}

Here a map $f : \alpha \to \beta$ is a \textit{length map} if $\ell_d(f(\alpha[u,v])) = \ell_d(\alpha[u,v])$
for all $u, v \in \alpha$, where $\alpha$ and $\beta$ are rectifiable curves with $\ell_d(\alpha) \leq \ell_d(\beta)$.

Observe that Definition~\ref{def:road}\,\eqref{road3} implies that $|x_i - x_j| \leqslant \nu$ for all $i$ and $j$. If $x_i = x_j = x$ for all $i$ and $j$, we say that $\bar{\alpha}$ is a \textit{road from $x$}. The \textit{locus} $|\bar{\alpha}|$ of a road $\bar{\alpha}$ is the union of all arcs $\alpha_i$.

In the case $\nu = 0$, $h = 0$, we have a geodesic ray. More precisely, the locus $|\bar{\alpha}|$ is a geodesic ray, and each $\alpha_i$ is an initial subarc.

According to \cite[Lemma~6.13]{Vaisala}, $\{y_i\}$ is a Gromov sequence and we can define $b=\hat{y}=[\{y_i\}]\in \partial_G X$, which will be introduced shortly below. Write $\bar{\alpha}:\bar{x}\curvearrowright b$ and say $\bar{\alpha}$ joining the sequence $\bar{x}=\{x_i\}$ to $b$. If for any $i$, $x_i=x$, then we write  $\bar{\alpha}:x\curvearrowright b$.

\subsection{Gromov hyperbolic spaces}
There are many equivalent definitions of Gromov hyperbolicity, which can be found in \cite{Vaisala}. By replacing the geodesic rays with roads, we can obtain the definition of a roughly starlike $\delta$-Gromov hyperbolic space.

\begin{defn}[Roughly starlike]\label{def:roughly starlike}
	Let $X$ be a $\delta$-Gromov hyperbolic space, and $H, \nu, h\geq 0$. $X$ is \textit{$(H, \nu, h)$-roughly starlike} with a point $w \in X$, if for each $x \in X$, there exists a $(\nu, h)$-road $\bar{\alpha}: w \curvearrowright a \in \partial_G X$, with $\dist(x, |\bar{\alpha}|) \leq H$.
\end{defn}

Next, we recall the definition of Rips space introduced in~\cite{Vaisala}.
\begin{defn}[Rips space]\label{def:rips}
	Let $h \geq 0$ be a constant. If there is a constant $C > 0$ such that for every $h$-short triangle $\Delta = (\alpha_1, \alpha_2, \alpha_3)$ in $(X,d)$  and for all $w \in \alpha_i$, $i \in \{1, 2, 3\}$, with $\alpha_4 = \alpha_1$ and $\alpha_5 = \alpha_2$, it holds
	\[
	\dist(w, \alpha_{i+1} \cup \alpha_{i+2}) \leq C,
	\]
	then we call $(X, d)$ a \emph{$(C, h)$-Rips space} and  the triangle \emph{$C$-thin}.
\end{defn}

Gromov hyperbolic space are Rips spaces, quantitatively. 
\begin{lem}[{\cite[Theorem 2.35]{Vaisala}}]\label{thm:rips}
	If $(X,d)$ is a $\delta$-Gromov hyperbolic, then it is a $(\delta'(\delta,h),h)$-Rips space with $\delta'(\delta,h) = 3\delta + \frac{3}{2}h$ for each $h > 0$.
\end{lem}

The following lemma provides quantitatively equivalent characterizations of Gromov hyperbolicity via quasigeodesic triangles.
\begin{lem}
	\label{lem:vaisala-package}
	Let $(X,\rho)$ be a length space and for $\mu>0$, let $\mathcal A$ be a family of
	$(1,\mu)$-quasigeodesics joining every pair of points of $X$. Then the following statements are quantitatively equivalent.
	\begin{enumerate}
		\item\label{lem:result1} Every triangle whose sides belong to $\mathcal A$ is $\vartheta_1$-thin.
		\item\label{lem:result2} $(X,\rho)$ is $\vartheta_2$-Gromov hyperbolic.
	\end{enumerate}
	In particular, if \eqref{lem:result1} or \eqref{lem:result2} holds, for every
	$\alpha:b\curvearrowright c$ in $\mathcal A$ and every $p\in X$, we have
	\begin{equation*}
		\rho(p,\alpha)-\vartheta_3
		\le (b|c)_p^\rho
		\le \rho(p,\alpha)+\vartheta_3.
	\end{equation*}
\end{lem}

\begin{proof}
	\eqref{lem:result1}$\Rightarrow$\eqref{lem:result2}: 
	Let $a,b,c,p\in X$ and choose
	$\alpha:b\curvearrowright c$, $\beta:a\curvearrowright c$, and
	$\gamma:a\curvearrowright b$ in $\mathcal A$.
	By the thinness assumption, for $x\in\alpha$,
	there is $y\in\beta\cup\gamma$ such that $\rho(x,y)\le \vartheta_1$. 
	
	We claim that $(a|c)_x^\rho\wedge(a|b)_x^\rho\le \vartheta_1+\mu/2$.
	
	Indeed, if $y\in\beta$,
	then
	\[
	\rho(a,x)+\rho(x,c)\leq\rho(a,y)+\rho(x,y)+\rho(x,y)+\rho(y,c)
	\le \ell_\rho(\beta)+2\vartheta_1
	\le \rho(a,c)+\mu+2\vartheta_1.
	\]
	Hence $(a|c)_x^\rho\le \vartheta_1+\mu/2$. Similarly, if $y\in \gamma$, then one can prove $(a|b)_x^\rho\le \vartheta_1+\mu/2$. This completes the proof of claim. 
	
	We now apply the preceding claim with $a=p$. Set $\Delta=\vartheta_1+\mu/2$, and 
	let
	\[
	E=\{z\in\alpha:(p|b)_z^\rho\le\Delta\},
	\quad
	F=\{z\in\alpha:(p|c)_z^\rho\le\Delta\}.
	\]
	The sets $E$ and $F$ are closed and $\alpha=E\cup F$ by the claim. Observe that $b\in E$ and $c\in F$.
	Since $\alpha$ is connected, $E\cap F\ne\varnothing$. If $z\in E\cap F$,
	then
	\beqq
	\begin{aligned}
		4\Delta
		&\ge 2(p|b)_z^\rho+2(p|c)_z^\rho=2\rho(p,z)+\rho(b,z)+\rho(c,z)-\rho(p,b)-\rho(p,c)\\
		&\geq2\rho(p,\alpha)+\rho(b,c)-\rho(p,b)-\rho(p,c) = 2\rho(p,\alpha)-2(b|c)_p^\rho.
	\end{aligned}
	\eeqq
	Consequently, we have 
	\[
	\rho(p,\alpha)\le(b|c)_p^\rho+2\vartheta_1+\mu.
	\]
	On the other hand, since $\alpha$ is $\mu$-short, \cite[Lemma 2.9]{Vaisala} gives
	\[
	(b|c)_p^\rho\le\rho(p,\alpha)+\frac \mu2.
	\]
	Therefore
	\begin{equation}\label{eq:vaisala-package}
		\rho(p,\alpha)-2\vartheta_1-\mu
		\le (b|c)_p^\rho
		\le \rho(p,\alpha)+\frac \mu2.
	\end{equation}
	For curves $\beta$ and $\gamma$, we also have the similar results with \eqref{eq:vaisala-package}. Using \eqref{eq:vaisala-package}, we obtain
	\beqq
	\begin{aligned}
		(a|c)_p^\rho\wedge(a|b)_p^\rho
		&\le \rho(p,\beta)\wedge\rho(p,\gamma)+\frac \mu2\\
		&\le \rho(p,\alpha)+\vartheta_1+\frac \mu2\\
		&\le (b|c)_p^\rho+3\vartheta_1+\frac{3\mu}{2}.
	\end{aligned}
	\eeqq
	This means that $(X,\rho)$ is $(3\vartheta_1+3\mu/2)$-Gromov hyperbolic.
	
	\eqref{lem:result2}$\Rightarrow$\eqref{lem:result1}: Gromov hyperbolic space is Rips space by Lemma~\ref{thm:rips}, and every $(1,\mu)$-quasigeodesic is $\mu$-short. Thus \eqref{lem:result1} holds.
	
\end{proof}

The next stability result shows that, in Gromov hyperbolic space, short curves and rough quasigeodesics with the same endpoints are close.
\begin{lem}[{\cite[Lemma 3.7]{Vaisala}}]\label{thm:stability}
Let $(X,d)$ be a $\delta$-Gromov hyperbolic space, $\alpha: x \curvearrowright y$ be an $h$-short arc and $\gamma: x \curvearrowright y$ be a $(c,\mu)$-rough quasigeodesic. Then
\[
d_H(\alpha,\gamma) \leq M(\delta,c,\mu,h),
\]
where $M(\delta,c,\mu,h)=1080c^4\delta+25c^2 \max\{h, \mu\}.$ 
\end{lem}


In the proof of our main results, we also need the following lemma.
\begin{lem}\label{lem:proj1}
Let $(X,d)$ be a $\delta$-Gromov hyperbolic space, $\alpha: x \curvearrowright y$ be an $h$-short arc with $h > 0$ and $\gamma: x \curvearrowright z$ be a $(c,\mu)$-rough quasigeodesic with $|y-z| \leq \tau$. Then for every $w \in \gamma$, there exists $v \in \alpha$ such that
\[
|w - v| \leq \delta'(\delta,h) + M(\delta,c,\mu,h) + h + \tau,
\]
where $\delta'(\delta,h)= 3\delta + \frac{3}{2}h$ is given by Lemma~\textup{\ref{thm:rips}}.
\end{lem}
\begin{proof}
Let $\beta_0: x \curvearrowright z$ and $\beta: y \curvearrowright z$ be $h$-short arcs with $h > 0$. By Lemma~\ref{thm:stability}, for every $w \in \gamma$ there exists $w' \in \beta_0$ with $|w-w'| \leq M(\delta,c,\mu,h)$.

Since $X$ is $\delta$-Gromov hyperbolic, by Lemma~\ref{thm:rips}, it is a $(\delta'(\delta,h),h)$-Rips space with $\delta'(\delta,h)=3\delta+\frac{3}{2}h$. Hence there exists $w'' \in \alpha \cup \beta$ such that $|w'-w''| \leq \delta'(\delta,h)$.

If $w'' \in \alpha$, set $v = w''$. Then $|w-v| \leq |w-w'| + |w'-w''| \leq M + \delta'(\delta,h)$.

If $w'' \in \beta$, set $v = y$. Then
\[
|w-v| \leq |w-w'| + |w'-w''| + |w''-y| \leq M + \delta'(\delta,h) + \ell_d(\beta) \leq M + \delta' (\delta,h)+ \tau + h.
\]
The proof is thus complete.
\end{proof}

To simplify our notation, from now on, we shall write $\delta_0=\delta'(\delta,2)=3\delta + 3$ and  $$M_0=M(\delta,c,\mu,2)=1080c^4\delta+25c^2\max\{2,\mu\}.$$
 
\begin{lem}\label{quasigeodesic ieq}
	Let $(X,d)$ be a $\delta$-Gromov hyperbolic space, $\gamma \in \Lambda_{x,y}^{c,\mu}(X,d)$ and $p \in X$. Then for any $z \in \gamma$, there exists an endpoint $a\in\{x,y\}$ such that for every $u\in\gamma[a,z]$,  
	$$|p-u|\geq |p-z| + |u-z| - \kappa,$$
 where $\kappa=\kappa(\delta,c,\mu) = 8\delta + 6M_0 + 2\delta_0 + 20$.
\end{lem}
\bpf
Let 
$\alpha: x \curvearrowright y$ be a $2$-short arc. 
For $z \in \gamma$, by Lemma~\ref{thm:stability}, there exists $z' \in \alpha$ such that  
\begin{equation}\label{2.1}
	|z-z'| \leq M_0. 
\end{equation}  
Let $\alpha=\alpha_x\cup \alpha^*\cup \alpha_y$ be the subdivision of $\alpha$ induced by the third vertex $p$ (see Definition \ref{def:induced subdivision}). Then $\ell_d(\alpha[x,x_\alpha]) = (y|p)_x$ and $\ell_d(\alpha[y_\alpha,y]) = (x|p)_y$. We first consider the case of $z' \in \alpha_x\cup\alpha^*= \alpha[x,y_\alpha]$.

For $u \in \gamma[x,z]$, by Lemma~\ref{lem:proj1} (with $\tau = M_0$, $h=2$ and $y=z'$), there exists $u_0 \in \alpha[x,z']$ such that
\begin{equation}\label{2.2}
	|u-u_0| \leq 2M_0 + \delta_0 + 2. 
\end{equation}	 
On the other hand, using standard distance estimate on short curves (see \cite[Lemma 4.1]{Allu}), we have 
\begin{equation}\label{2.3}
	|p-u_0| \geq |p-z'| + |u_0-z'| - 8\delta - 16.
\end{equation}
Combining the above two estimates leads to 
\beqq
\begin{aligned}
	|p-u| &\geq |p-u_0| - |u-u_0| \\
	&\stackrel{\eqref{2.2}+\eqref{2.3}}{\geq} |p-z'| + |u_0-z'| - 8\delta - 16 - (2M_0 + \delta_0 + 2) \\
	&\stackrel{\eqref{2.1}}{\geq} |p-z| - M_0 + |u_0-z| - M_0 - 8\delta - 18 - 2M_0 - \delta_0 \\
	&\stackrel{\eqref{2.2}}{\geq} |p-z| + |u-z| - \kappa,
\end{aligned}
\eeqq
where $\kappa=\kappa(\delta,c,\mu) = 8\delta + 6M_0 + 2\delta_0 + 20$.

Symmetrically, if $z' \in \alpha_y\cup\alpha^*=\alpha[y,x_\alpha]$, for $u \in \gamma [z, y]$, we also have $$|p-u|\geq |p-z| + |u-z| - \kappa.$$
This completes the proof of Lemma~\ref{quasigeodesic ieq}.
\epf

Finally, we need an extension of the Projection Lemma of V\"ais\"al\"a \cite[Lemma 3.6]{Vaisala} to $(c, \mu)$-quasigeodesics. 
\begin{lem}[Generalized Projection Lemma]\label{lem:proj2}
Suppose $(X,d)$ is $\delta$-Gromov hyperbolic. Let $\gamma \in \Lambda_{x,y}^{c,\mu}(X,d)$ and $x_1, x_2 \in X$, $y_1, y_2 \in \gamma$ be points such that
\begin{enumerate}\label{arabic*}
\item $|x_i - y_i| = \dist(x_i, \gamma) \geq r > 0$ for $i = 1, 2$;
\item $|x_1 - x_2| < 2r - (4\delta_0 + 2M_0 + 2)$.
\end{enumerate}
Then $|y_1 - y_2| \leq 8\delta_0 + 4M_0 + 4$.
\end{lem}

\begin{proof}
Let 
$\alpha: y_1 \curvearrowright y_2$, $\alpha_1: x_1 \curvearrowright y_1$, $\alpha_2: x_2 \curvearrowright y_2$, and $\alpha_3: x_1 \curvearrowright x_2$ be $2$-short curves. Select $y_0 \in \gamma[y_1, y_2]$ with $|y_0 - y_1| = |y_0 - y_2|$. Then Lemma~\ref{thm:stability} implies that there exists $y_0' \in \alpha$ such that 
$$|y_0-y_0'| \leq M_0.$$ 
Since $X$ is a $\delta$-Gromov hyperbolic, by Lemma~\ref{thm:rips}, it is a $(\delta_0,2)$-Rips space. Hence there exists $x_0 \in \alpha_1 \cup \alpha_2 \cup \alpha_3$ such that $|x_0-y_0'| \leq 2\delta_0$.
\smallskip 

\textbf{Case 1:} $x_0 \in \alpha_3$. 
\smallskip 

In this case, we have 
\beqq
\begin{aligned}
r &\leq \dist(x_1, \gamma) \leq |x_1 - y_0| \\
&\leq |x_1 - x_0| + |x_0 - y_0'| + |y_0' - y_0| \\
&\leq |x_1 - x_0| + 2\delta_0 + M_0.
\end{aligned}
\eeqq
Similarly, $r \leq |x_2 - x_0| + 2\delta_0 + M_0$. It follows that 
\[
2r \leq |x_1 - x_0| + |x_2 - x_0| + 4\delta_0 + 2M_0 \leq |x_1 - x_2| + 2 + 4\delta_0 + 2M_0 < 2r,
\]
which is a contradiction.
\smallskip 

\textbf{Case 2:} $x_0 \in \alpha_1\cup \alpha_2$. 
\smallskip 

By symmetry, we may assume that $x_0\in \alpha_1$. Then
\beqq
\begin{aligned}
|x_1 - y_1| &= \dist(x_1, \gamma) \leq |x_1 - y_0| \\
&\leq |x_1 - x_0| + |x_0 - y_0'| + |y_0' - y_0| \\
&\leq |x_1 - x_0| + 2\delta_0 + M_0.
\end{aligned}
\eeqq
Since $|x_0 - x_1| + |x_0 - y_1| \leq \ell_d(\alpha_1) \leq |x_1 - y_1| + 2$, substituting in the previous estimate gives 
$$|x_0 - y_1| \leq 2\delta_0 + M_0 + 2.$$ 
Consequently, we infer
\beqq
\begin{aligned}
|y_1 - y_2| &\leq 2|y_1 - y_0| \leq 2(|y_1 - x_0| + |x_0 - y_0'| + |y_0' - y_0|) \\
&\leq 2(2\delta_0 + M_0 + 2 + 2\delta_0 + M_0) \\
&= 8\delta_0 + 4M_0 + 4.
\end{aligned}
\eeqq

The proof is complete. 
\end{proof}

\subsection{Gromov boundary}
In this section, we recall the definition of Gromov boundary. In a
proper geodesic Gromov hyperbolic metric space, one could define the Gromov
boundary as the set of equivalence classes of geodesic rays. Since we
are interested in spaces which are merely length, we
follow Väisälä's definition~\cite{Vaisala} using Gromov sequences.
The two definitions agree in proper geodesic hyperbolic spaces.

Let $X$ be a metric space and fix $p\in X$. A sequence
$\{x_n\}\subset X$ is called a \emph{Gromov sequence} if
$(x_i\mid x_j)_p\to\infty$
as  $i,j\to\infty.$
Two Gromov sequences $\{x_n\}$ and $\{y_n\}$ are said to be
\emph{equivalent} if
$(x_i\mid y_i)_p\to\infty$
as  $i\to\infty.$
In Gromov hyperbolic space, this relation is an equivalence relation. Let $\widehat{x}$ denote the equivalence class containing the Gromov
sequence $\{x_n\}$. The \emph{Gromov boundary} of $X$ is defined by
\begin{equation*}
	\partial_G X
	=
	\left\{
	\widehat{x}:
	\{x_n\}\text{ is a Gromov sequence in }X
	\right\}.
\end{equation*}
For proper geodesic Gromov hyperbolic spaces, it is well-known (see \cite[Section 3 and 4]{BHK}) that there exists a metric $d_{\varepsilon,p}$ on $\partial_GX$ such that the metric spaces $(\partial_GX,d_{\varepsilon,p})$ and $(\partial X_{\varepsilon},d_{\varepsilon})$ are quasiisometric. Consequently, for any $z\in X$ and any $\zeta\in \partial X_{\varepsilon}$, there exists a unique $[\{x_n\}]\in \partial_G X$ corresponding to $\zeta$ so that $x_n\stackrel{d_\varepsilon}{\longrightarrow}\zeta,$ when $n\rightarrow \infty$, and that $|z-x_n|\xrightarrow{n\to\infty}\infty$. 

This result holds in our setting as well. 
\begin{prop}\label{prop:boundary-id}
	Let $X$ be a complete $\delta$-Gromov hyperbolic space and $\varepsilon_0=\varepsilon_0(\delta,1,1,1)$ be the constant given by Theorem~\ref{thm:G-H}.
	Then for every $0<\varepsilon\leq\varepsilon_0$, there is a natural identification
	\[
	\Phi:\left(\partial_G X, d_{\varepsilon,p}\right)\longrightarrow
	(\partial X_\varepsilon,  d_\varepsilon)
	\]
	from the Gromov boundary to the metric boundary of $X_\varepsilon$.
	Moreover, $\Phi$ is a quasiisometry. Indeed, there exists $C_0=C_0(\delta)\geq1$ such that for all $\xi,\eta\in\partial_G X$, it holds 
	\begin{equation*}
		\frac{1}{C_0\varepsilon}\,d_{\varepsilon,p}(\xi,\eta)
		\leq d_\varepsilon(\Phi(\xi),\Phi(\eta))
		\leq \frac{C_0}{\varepsilon}\,d_{\varepsilon,p}(\xi,\eta).
	\end{equation*} 
\end{prop}

We would also like to point out that in \cite[Theorem~1.3]{AlluJose}, a version of Proposition \ref{prop:boundary-id} was claimed; however, the proof seems to be incomplete and requires an additional assumption to guarantee surjectivity. As the proof of Proposition \ref{prop:boundary-id} is very similar to that used in \cite[Proposition 4.13]{BHK}, we present it in the appendix. 



\section{The Gehring-Hayman inequality for quasigeodesics}\label{sec:Gehring-Hayman}

In this section, let $X=(X,d)$ be a $\delta$-Gromov hyperbolic space. First, we prove an auxiliary theorem, which roughly says that if a $(c,\mu)$-quasigeodesic $\gamma$ stays close to a rectifiable curve $\widetilde\alpha$ with the same endpoints, then $\ell_\rho(\gamma)$ is controlled quantitatively from above by $\ell_\rho(\widetilde\alpha)$.

\begin{thm}\label{thm:weighted-close}
Fix $\varepsilon_0>0$, $c\geq1$, $\mu>0$ and $L>0$. Let $\rho: X \to (0,\infty)$ be a continuous density satisfying \eqref{eq:Harnack-general} for some constants $C\geq 1$ and $0<\varepsilon\leq\varepsilon_0$. Then 
 there exists a constant $K=K(\varepsilon_0, C, c,\mu,L)$ such that for each  $\gamma \in \Lambda_{x,y}^{c,\mu}(X,d)$ and each rectifiable curve $\widetilde\alpha:x\curvearrowright y$ with
\begin{equation}\label{eq:close-main}
    \gamma\subset B_d(\widetilde\alpha,L),
\end{equation}
it holds 
 $ \ell_\rho(\gamma)\leq K\ell_\rho(\widetilde\alpha)$.
\end{thm}

\begin{proof}

We consider two cases. 
\smallskip 

\textbf{Case 1:} $|x-y|\leq4cL+\mu$. 
\smallskip 

In this case, we have 
\beqq
\begin{aligned}
\ell_\rho(\gamma)=\int_{\gamma} \rho \, ds \stackrel{\eqref{eq:Harnack-general}}{\leq} C\rho(x) \int_0^{\ell_d(\gamma)} e^{\varepsilon t} \, dt &\stackrel{\eqref{eq:prop of curve}}{\leq} C\rho(x) \cdot e^{\varepsilon (c+\mu)|x-y|}\cdot (c+\mu)|x-y|\\
&  \leq (c+\mu)C\rho(x) e^{\varepsilon_0 (c+\mu)(4cL+\mu) } |x-y| .
\end{aligned}
\eeqq
For the lower bound on $\ell_\rho(\widetilde\alpha)$, select a subcurve $\widetilde\alpha' \subset \widetilde\alpha$ starting from $x$ with  $\ell_d(\widetilde\alpha') = |x-y|$. Then
\[
\ell_\rho(\widetilde\alpha) \geq \ell_\rho(\widetilde\alpha')=\int_{\widetilde\alpha'} \rho \, ds \stackrel{\eqref{eq:Harnack-general}}{\geq} \frac{1}{C}\rho(x) e^{-\varepsilon_0|x-y|} |x-y| \geq \frac{1}{C}\rho(x) e^{-\varepsilon_0(4cL+\mu)}|x-y|.
\]
Combining the above two estimates gives 	
$$\ell_\rho(\gamma) \leq (c+\mu)C^2 e^{\varepsilon_0 (c+\mu+1)(4cL+\mu) }\ell_\rho(\widetilde\alpha).$$

\textbf{Case 2:} $|x-y|>4cL+\mu$.
\smallskip 

Choose a sequence of points $x = x_1, \cdots, x_n, x_{n+1} = y$, $n\geq2$, on $\gamma$ such that $\ell_d(\gamma[x_k, x_{k+1}]) = 4cL+\mu$ for $1\leq k < n$ and $\ell_d(\gamma[x_n, x_{n+1}]) < 4cL+\mu$. 

We claim that $B_d(x_{k_1}, 2L) \cap B_d(x_{k_2}, 2L) = \emptyset$ for $1 \leq k_1 < k_2 \leq n$.

If not, then there exists $z \in B_d(x_{k_1}, 2L) \cap B_d(x_{k_2}, 2L)$. On the one hand, by \eqref{eq:prop of subcurve}, we have 
 $$|x_{k_1} - x_{k_2}| \geq \frac{1}{c}\left(\ell_d(\gamma[x_{k_1}, x_{k_2}])-\mu\right)\geq 4L.$$ 
On the other hand, we have 
 $$|x_{k_1} - x_{k_2}| \leq |x_{k_1} - z| + |z - x_{k_2}| < 4L,$$
 which is clearly a contradiction. This completes the proof of claim.

By \eqref{eq:close-main}, $\widetilde\alpha$ meets each $B_d(x_k,L)$ with $k\in\{1,2,\cdots,n \}$. Note that $\widetilde\alpha$ is not contained in $B_d(x_k,2L)$. In fact, if $\widetilde\alpha\subset B_d(x_k, 2L)$, then $\{x, y \}\subset B_d(x_k, 2L)$ and so $|x-y|< 4L$, which is a contradiction. Thus there is a subcurve $\widetilde\alpha_k$ of $\widetilde{\alpha}$ contained in $\overline{B}_d(x_k,2L)\setminus B_d(x_k,L)$ joining $B_d(x_k,L)$ to $X\setminus B_d(x_k,2L)$. The curves $\widetilde\alpha_k$ are pairwise disjoint with $\ell_d(\widetilde\alpha_k)\geq L$. 

For each $k\in\{1,2,\cdots,n \}$, we have 
\[
\ell_\rho(\widetilde\alpha_k)=\int_{\widetilde\alpha_k} \rho \, ds \stackrel{\eqref{eq:Harnack-general}}{\geq} \frac{1}{C}\rho(x_k) \int_{\widetilde\alpha_k} e^{-\varepsilon|x_k - \widetilde\alpha_k(t)|} \, dt \geq \frac{1}{C}Le^{-2\varepsilon_0 L}\rho(x_k).
\]
On the other hand, for $\gamma_k=\gamma[x_k,x_{k+1}]$, we have
\[
\ell_\rho(\gamma_k)=\int_{\gamma_k} \rho \, ds \stackrel{\eqref{eq:Harnack-general}}{\leq} C\rho(x_k) \int_{\gamma_k} e^{\varepsilon|x_k - \gamma_k(t)|} \, dt \leq C(4cL+\mu) e^{\varepsilon_0 (4cL+\mu)}\rho(x_k).
\]
It follows from the above two estimates that 
$$\ell_\rho(\gamma_k) \leq C^2(4c+\frac{\mu}{L}) e^{\varepsilon_0 (4cL+2L+\mu)} \ell_\rho(\widetilde\alpha_k).$$ 
Summing over $k$, we obtain 
\beqq
\begin{aligned}
	\ell_\rho(\gamma) = \sum_{k=1}^n \ell_\rho(\gamma_k) &\leq C^2(4c+\frac{\mu}{L}) e^{\varepsilon_0 (4cL+2L+\mu)} \sum_{k=1}^n\ell_\rho(\widetilde\alpha_k)\\
	&\leq C^2(4c+\frac{\mu}{L}) e^{\varepsilon_0 (4cL+2L+\mu)} \ell_\rho(\widetilde\alpha). 
\end{aligned}
\eeqq

In both cases, there exists $K=K(\varepsilon_0, C, c,\mu,L)$ such that  $ \ell_\rho(\gamma)\leq K\ell_\rho(\widetilde\alpha)$.
\qedhere
\end{proof}

Following a similar strategy as Bonk-Heinonen-Koskela \cite[Section 5]{BHK}, we shall apply Theorem~\ref{thm:weighted-close} to prove the Gehring-Hayman inequality in Gromov hyperbolic spaces. To be more precise,  given a $(c,\mu)$-quasigeodesic $\gamma\in \Lambda_{x,y}^{c,\mu}(X,d)$ and a rectifiable curve $\alpha:x\curvearrowright y$, we shall construct a rectifiable curve $\widetilde\alpha:x\curvearrowright y$ with $\ell_\rho(\widetilde{\alpha})\leq \ell_\rho(\alpha)$ such that $\gamma\subset B_d(\widetilde\alpha,L)$. 

For this purpose, we first prove  the following result, which roughly says that given any rectifiable curve $\alpha:x\curvearrowright y$ of finite conformal length, we may find a ``nearly" rough quasigeodesic  $\widetilde{\alpha}: x \curvearrowright y$ with $\ell_\rho(\widetilde{\alpha}) \leq \ell_\rho(\alpha)$. The proof is similar to that used in \cite[Lemmas 5.7 and 5.8]{BHK}. 

\begin{prop}\label{prop:local}
	Let $\alpha: x \curvearrowright y$ be rectifiable with $\ell_\rho(\alpha) < \infty$, $c\geq 1$ and $\mu> 0$. Then there exists $\widetilde{\alpha}: x \curvearrowright y$ with $\ell_\rho(\widetilde{\alpha}) \leq \ell_\rho(\alpha)$ such that
	\[
	\ell_d(\widetilde{\alpha}[u,v]) \leq 3(c+\mu)C^2|u-v| + 1
	\]
	whenever $u,v \in \widetilde{\alpha}$, $u \in \widetilde{\alpha}[x,v]$, and $|u-v| \leq \frac{1}{12\varepsilon (c+\mu)C^2}$.
\end{prop}

\begin{proof}
	Let $\Gamma = \{\beta: x \curvearrowright y \mid \ell_\rho(\beta) \leq \ell_\rho(\alpha)\}$ and $S = \inf\limits_{\beta \in \Gamma} \ell_d(\beta) < \infty$. Choose $\widetilde{\alpha} \in \Gamma$ with $\ell_d(\widetilde{\alpha}) \leq S + \frac{1}{2}$. We shall prove that $\widetilde{\alpha}$ is the right curve with desired property. 
	
	 Suppose for contradiction that there are $u,v \in \widetilde{\alpha}$, $u \in \widetilde{\alpha}[x,v]$ with $|u-v| \leq \frac{1}{12\varepsilon (c+\mu)C^2}$, but 
	 \begin{equation}\label{eq:for lower bound on tilde alpha}
	 \ell_d(\widetilde{\alpha}[u,v]) > 3(c+\mu)C^2|u-v| + 1.
	 \end{equation}
	 
	 We claim that for each $\gamma_0 \in \Lambda_{u,v}^{c,\mu}(X,d)$, it holds $\ell_\rho(\gamma_0) \leq \ell_\rho(\widetilde{\alpha}[u,v])$. Suppose we were able to prove this claim. Then the curve $\gamma_0' = \widetilde{\alpha}[x,u] \cup \gamma_0 \cup \widetilde{\alpha}[v,y]\in \Gamma$ with 
	\beqq
	\begin{aligned}
		\ell_d(\gamma_0') &= \ell_d(\widetilde{\alpha}) - \ell_d(\widetilde{\alpha}[u,v]) + \ell_d(\gamma_0) \\
		&< S + \tfrac{1}{2} - 3(c+\mu)C^2|u-v| - 1 + (c+\mu)|u-v| \\
		&\leq S - \tfrac{1}{2},
	\end{aligned}
	\eeqq
	which clearly contradicts with the definition of $S$. Thus it remains to prove the above claim. 
	
	For each $\gamma_0 \in \Lambda_{u,v}^{c,\mu}(X,d)$, by \eqref{eq:prop of curve}, $\ell_d(\gamma_0) \leq (c+\mu)|u-v|$ and so 
	\[
	\ell_\rho(\gamma_0) = \int_{\gamma_0} \rho \, ds\stackrel{\eqref{eq:Harnack-general}}{\leq} C\rho(u) \int_0^{\ell_d(\gamma_0)} e^{\varepsilon t} \, dt \leq C(c+\mu)\rho(u)  e^{\varepsilon (c+\mu)|u-v|} |u-v|.
	\]
	For the lower bound on $\ell_d(\widetilde{\alpha}[u,v])$, using \eqref{eq:for lower bound on tilde alpha} and the elementary inequality $1-e^{-t} \geq \frac{2}{3}t$ for $0 \leq t \leq \frac{1}{4}$, we infer that 
	\beqq
	\begin{aligned}
		\ell_\rho(\widetilde{\alpha}[u,v]) &= \int_{\widetilde{\alpha}[u,v]} \rho \, ds
		\stackrel{\eqref{eq:Harnack-general}}{\geq} \frac{1}{C}\rho(u) \int_0^{3(c+\mu)C^2|u-v|} e^{-\varepsilon t} \, dt \\
		&= \frac{1}{C\varepsilon}\rho(u)\big(1 - e^{-3\varepsilon (c+\mu)C^2|u-v|}\big) \geq 2C(c+\mu)\rho(u)|u-v|.
	\end{aligned}
	\eeqq
	Since $e^{1/(12C^2)} < 2$ for $C \geq 1$, we obtain from the above two estimates that $\ell_\rho(\gamma_0) \leq \ell_\rho(\widetilde{\alpha}[u,v])$. The proof of Proposition \ref{prop:local} is thus complete. 
\end{proof}

Next, we prove that in Gromov hyperbolic spaces, a ``nearly" rough quasigeodesic $\widetilde{\alpha}$ as in the previous proposition always stays in a neighborhood of a $(c, \mu)$-quasigeodesic $\gamma$ with the same end points.  
\begin{prop}\label{prop:curve-approx}
Let $(X,d)$ be $\delta$-Gromov hyperbolic and fix $c, \lambda\geq 1$, $\mu, \xi > 0$. Then there exist $N=N(\lambda,\delta,c,\mu,\xi)$ and $L_0=2\lambda N + 2\xi + 1$ such that for any  $\gamma \in \Lambda_{x,y}^{c,\mu}(X,d)$ and any rectifiable curve  $\widetilde{\alpha}: x \curvearrowright y$, if $\ell_d(\widetilde{\alpha}[u,v]) \leq \lambda|u-v| + \xi$ whenever $u,v \in \widetilde{\alpha}$, $|u-v| \leq L_0$, then 
$$\widetilde{\alpha} \subset B_d(\gamma, L_0).$$
\end{prop}

\begin{proof}
	Set 
	\[
	R = R(\lambda,\delta,c,\mu)=2 + 4\delta_0 + 4M_0 + \lambda(4\delta_0 + 2M_0 + 2),
	\]
	\[
	N=N(\lambda,\delta,c,\mu,\xi)= \big(1 + \lambda(8\delta_0 + 4M_0 + 4)\big)\big(2R + (8\delta_0 + 4M_0 + 4)\big) + (8\delta_0 + 4M_0 + 4)\xi,
	\]
	and $L_0=2\lambda N + 2\xi + 1$.
	
	Given $\widetilde{\alpha}$ as in the proposition, we may assume $\widetilde{\alpha}$ is parametrized by $d$-arc-length and
	define $f: [0, \ell_d(\widetilde{\alpha})] \to [0,\infty)$ by $f(t) = \dist(\widetilde{\alpha}(t), \gamma)$. Then $f(0) = f(\ell_d(\widetilde{\alpha})) = 0$. For $s,t \in [0, \ell_d(\widetilde{\alpha})]$,
	\[
	|f(s) - f(t)| \leq |\widetilde{\alpha}(s) - \widetilde{\alpha}(t)| \leq |s-t|.
	\]
	
	Set $k = \frac{1}{2}L_0$ and 
	$$M_k = \{(s,t) \in [0,\ell_d(\widetilde{\alpha})] \times [0,\ell_d(\widetilde{\alpha})] : s \leq t \text{ and } k \leq f(s) = f(t) \leq f(p) \text{ for } s \leq p \leq t\}.$$  For $(s,t)\in M_k$ and all $p\in [s,t]$, we have
	\begin{equation}\label{3.4}
	\dist(\widetilde{\alpha}(p), \gamma) \geq \dist(\widetilde{\alpha}(s), \gamma) = \dist(\widetilde{\alpha}(t), \gamma) \geq k = \tfrac{1}{2}L_0 \geq R.
	\end{equation}
	
	\textbf{Claim:} either $|\widetilde{\alpha}(s) - \widetilde{\alpha}(t)| > L_0$ or $|\widetilde{\alpha}(s) - \widetilde{\alpha}(t)| \leq N$.
	\smallskip 
	
	To prove the above claim, write $x'=\widetilde{\alpha}(s), y'=\widetilde{\alpha}(t)$ and take $x_0, y_0 \in \gamma$ such that 
	$$|x'-x_0| = |y'-y_0| = \dist(\widetilde{\alpha}[x', y'],\gamma) \geq R.$$ Let $\beta_1: x' \curvearrowright x_0$, $\beta_2: y' \curvearrowright y_0$ be $2$-short curves. Choose $x_1 \in \beta_1$, $y_1 \in \beta_2$ such that $|x'-x_1| = |y'-y_1| = R$ and select $\gamma_3 \in \Lambda_{x_1,y_1}^{c,\mu}(X,d)$. We first show that 
	\begin{equation}\label{eq:lower bound on alpha gamma3}
	\dist(\widetilde{\alpha}[x', y'], \gamma_3) \geq r := R - 2\delta_0 - 2M_0 = 2 + 2\delta_0 + 2M_0 + \lambda(4\delta_0 + 2M_0 + 2).
	\end{equation}
	
	Suppose for contradiction that there exists $z \in \widetilde{\alpha}[x', y']$, $z_1 \in \gamma_3$ such that $|z-z_1| < r$. Let $\beta_3: x_1 \curvearrowright y_1$, $\beta_4: x_0 \curvearrowright y_0$ be $2$-short curves. Thus by Lemma~\ref{thm:stability}, there exists $z_1' \in \beta_3$ such that 
	\begin{equation}\label{3.1}
		|z_1 - z_1'| \leq M_0.
	\end{equation}
	Since $X$ is a $(\delta_0,2)$-Rips space by Lemma~\ref{thm:rips}, there exists $z_0' \in \beta_1[x_0,x_1] \cup \beta_2[y_0,y_1] \cup \beta_4$ such that 
	\begin{equation}\label{3.2}
		|z_1' - z_0'| \leq 2\delta_0
	\end{equation}
	We consider two cases according to the position of $z_0'$.
	\smallskip 
	
	\textbf{Case 1:} $z_0' \in \beta_4$. 
	\smallskip 
	
	In this case, there exists $z_0 \in \gamma$ such that $|z_0' - z_0| \leq M_0$. By \eqref{3.4}, \eqref{3.1} and \eqref{3.2}, we have
	\beqq
	\begin{aligned}
		R &\leq |z - z_0| \leq |z-z_1| + |z_1-z_1'| + |z_1'-z_0'| + |z_0'-z_0| < r + 2M_0 + 2\delta_0 = R,
	\end{aligned}
	\eeqq
	which is a contradiction.
	\smallskip 
	
	\textbf{Case 2:} $z_0' \in \beta_1[x_0,x_1]\cup \beta_2[y_0,y_1]$. 
	\smallskip 
	
	By symmetry, we may assume $z_0' \in \beta_1[x_0,x_1]$. Then by \eqref{3.1} and \eqref{3.2}, we have
	\beqq
	\begin{aligned}
		|x'-x_0| &= \dist(\widetilde{\alpha}[x', y'],\gamma) \leq \dist(z,\gamma) \leq |z-x_0| \\
		&\leq |z-z_1| + |z_1-z_1'| + |z_1'-z_0'| + |z_0'-x_0| \\
		&< r + M_0 + 2\delta_0 + |x_1-x_0| + 2 < R - 2 + |x_1-x_0| \\
		&= |x'-x_1| + |x_1-x_0| - 2\leq \ell_d(\beta_1) - 2 \leq |x'-x_0|,
	\end{aligned}
	\eeqq
	which again is a contradiction. 
	
	In both cases, we have proved \eqref{eq:lower bound on alpha gamma3}. 
	
	Next, choose a sequence of points $x' = w_0, w_1, \cdots, w_{n-1}, w_n = y'$ on $\widetilde{\alpha}[x', y']$ so that 
	\[
	\ell_d(\widetilde\alpha[w_k, w_{k+1}]) = 1 + \lambda(8\delta_0 + 4M_0 + 4), \quad 0 \leq k \leq n-2,
	\]
	and $$\ell_d(\widetilde\alpha[w_{n-1}, w_n]) \leq 1 + \lambda(8\delta_0 + 4M_0 + 4).$$
	Then for each $0 \leq k \leq n-1$,  $|w_k - w_{k+1}|\leq \ell_d(\widetilde\alpha[w_k, w_{k+1}]) < 2r - (4\delta_0 + 2M_0 + 2)$.
	
	Based on \eqref{eq:lower bound on alpha gamma3}, we may choose $v_k \in \gamma_3$ such that $|w_k - v_k| = \dist(w_k, \gamma_3) \geq r$. Then by Lemma~\ref{lem:proj2}, $|v_k - v_{k+1}| \leq 8\delta_0 + 4M_0 + 4$ for all $k$. Note that $|w_0 - v_0| = \dist(x', \gamma_3) \leq |x'-x_1| = R$ and $|w_n - v_n| = \dist(y', \gamma_3) \leq |y'-y_1| = R$.
	
	If $|x'-y'| \leq L_0$, then
	$\ell_d(\widetilde\alpha[x', y']) \leq \lambda|x'-y'| + \xi$ by assumption and so
	\[
	n \leq \frac{\ell_d(\widetilde\alpha[x', y'])}{1 + \lambda(8\delta_0 + 4M_0 + 4)} + 1 \leq \frac{\lambda|x'-y'| + \xi}{1 + \lambda(8\delta_0 + 4M_0 + 4)} + 1.
	\]
	It follows from the above estimate that 
	\beqq
	\begin{aligned}
		|x'-y'| &= |w_0 - w_n| \leq |w_0 - v_0| + \sum_{k=0}^{n-1}|v_k - v_{k+1}| + |v_n - w_n| \leq 2R + n(8\delta_0 + 4M_0 + 4) \\
		&\leq 2R + (8\delta_0 + 4M_0 + 4) + \frac{(8\delta_0 + 4M_0 + 4)\xi}{1+\lambda(8\delta_0 + 4M_0 + 4)} + \frac{\lambda(8\delta_0 + 4M_0 + 4)}{1+\lambda(8\delta_0 + 4M_0 + 4)}|x'-y'|,
	\end{aligned}
	\eeqq
 and thus 
 $$|x'-y'|\leq  \big(1 + \lambda(8\delta_0 + 4M_0 + 4)\big)\big(2R + (8\delta_0 + 4M_0 + 4)\big) + (8\delta_0 + 4M_0 + 4)\xi=N.$$ 
This completes the proof of claim.

	If $|\widetilde\alpha(s) - \widetilde\alpha(t)| > L_0$, then $2k = L_0 < |\widetilde\alpha(s) - \widetilde\alpha(t)| \leq  |s-t|$.
	
	If $|\widetilde\alpha(s) - \widetilde\alpha(t)| \leq N$, then by the assumption of Proposition~\ref{prop:curve-approx}, 
	$$|s-t| \leq \lambda|\widetilde\alpha(s) - \widetilde\alpha(t)| + \xi \leq \lambda N + \xi = \frac{1}{2}(L_0-1) < k.$$
	
	Consequently, $|s-t|<k$ or $|s-t|>2k$. Then by \cite[Lemma 5.15]{BHK}, it holds 
	$$\max\limits_{x \in [0,\ell_d(\widetilde\alpha)]} f(x) < \frac{3}{2}k = \frac{3}{4}L_0 < L_0.$$ 
	This implies $\widetilde\alpha \subset B_d(\gamma, L_0)$ as desired.
\end{proof}

Finally, we need a simple technical lemma, which provides a sufficient condition ensuring that a $(c, \mu)$-quasigeodesic $\gamma$ is close to $\widetilde{\alpha}$. 
\begin{lem}\label{lem:reverse}
Fix $c\geq 1$, $\mu, r > 0$ and $S\geq 0$. Suppose $\gamma \in \Lambda_{x,y}^{c,\mu}(X,d)$ and $Q \subset X$ is a set with $\{x,y\} \subset Q \subset \overline{B_d}(\gamma,r)$ such that  $\dist(Q_1, Q_2) \leq S$ whenever $Q = Q_1 \cup Q_2$, $x \in Q_1$, $y \in Q_2$. Then $\gamma \subset B_d(Q, c(2r+S)+\mu)$.
\end{lem}

\begin{proof}
If not, then there exist $\varsigma>0$ and $z \in \gamma$ such that $\dist(z,Q) \geq c(2r+S+4\varsigma)+\mu$. Then 
$$\ell_d(\gamma[x,z]) \geq |x-z| \geq \dist(z,Q) \geq c(2r+S+4\varsigma)+\mu.$$ 
Thus there exists $x_1 \in \gamma[x,z]$ such that $\ell_d(\gamma[x_1,z]) = c(r + \frac{S}{2} + 2\varsigma)+\mu$. Similarly, there exists $x_2 \in \gamma[z,y]$ such that $\ell_d(\gamma[z,x_2]) = c(r + \frac{S}{2} + 2\varsigma)+\mu$.

Set $\gamma' = \gamma[x_1,x_2]$, $\gamma_1 = \gamma[x,x_1]$, $\gamma_2 = \gamma[x_2,y]$. Let $U' = B_d(\gamma', r+\frac{\varsigma}{2})$ and $U_i = B_d(\gamma_i, r+\frac{\varsigma}{2})$ for $i=1,2$.

If $U' \cap Q \neq \emptyset$, then there exists $z' \in \gamma'$ such that $\dist(z',Q) < r+\frac{\varsigma}{2}$. Thus
\[
c(2r+S+4\varsigma)+\mu \leq d(z,Q) < |z-z'| + r+\frac{\varsigma}{2} \leq c(r+\frac{S}{2}+2\varsigma)+\mu+r+\frac{\varsigma}{2},
\]
which is a contradiction.

If $U' \cap Q = \emptyset$, then since  $Q \subset \overline{B_d}(\gamma,r)$, we have $Q = Q_1 \cup Q_2$, where  $Q_i = Q \cap U_i$. Thus there exists $q_i \in Q_i$ such that $|q_1 - q_2| \leq S + \frac{\varsigma}{2}$, and $y_i \in \gamma_i$ with $|y_i - q_i| < r+\frac{\varsigma}{2}$. Hence 
$$|y_1 - y_2|\leq |y_1-q_1|+|q_1-q_2|+|q_2-y_2| < 2r + S + \frac{3}{2}\varsigma.$$ 
It follows from this and \eqref{eq:prop of subcurve} that
\[
c(2r+S+4\varsigma)+2\mu = \ell_d(\gamma') \leq \ell_d(\gamma[y_1,y_2]) \leq c|y_1-y_2|+\mu< c(2r+S+\frac{3}{2}\varsigma)+\mu,
\]
which is a contradiction. 

This contradiction shows that  $\gamma \subset B_d(Q, c(2r+S)+\mu)$.
\end{proof}

With all these results at hand, we are able to prove Theorem~\textup{\ref{thm:G-H}}.
\begin{proof}[Proof of Theorem~\textup{\ref{thm:G-H}}]
For $\lambda = 3(c+\mu)C^2$ and $\xi = 1$, let $N=N(\lambda,\delta,c,\mu,\xi)=N(c,C,\delta,\mu)$ and $L_0=L_0(c,C,\delta,\mu)=2\lambda N + 2\xi + 1$ be given by Proposition \ref{prop:curve-approx}. Then we set 
\[
\varepsilon_0 =\varepsilon_0(\delta,C,c,\mu)= \frac{1}{13L_0(c+\mu)C^2}. 
\]	
		
Fix $\gamma \in \Lambda_{x,y}^{c,\mu}(X,d)$ and a rectifiable curve $\alpha:x\curvearrowright y$ with $\ell_\rho(\alpha)<\infty$. In fact, if $\ell_\rho(\alpha)=\infty$, there is nothing to prove. By Proposition~\ref{prop:local}, there exists a rectifiable curve $\widetilde{\alpha}:x\curvearrowright y$ such that
$
\ell_\rho(\widetilde{\alpha})\le \ell_\rho(\alpha),
$
and $\widetilde{\alpha}$ satisfies 
\[
\ell_d(\widetilde{\alpha}[u,v]) \leq 3(c+\mu)C^2|u-v| + 1=\lambda|u-v| + 1
\]
whenever $u,v \in \widetilde{\alpha}$, $u \in \widetilde{\alpha}[x,v]$, and $|u-v| \leq \frac{1}{12\varepsilon(c+\mu) C^2}.$ Since $0<\varepsilon\le\varepsilon_0$, we have
\[
 \frac{1}{12\varepsilon(c+\mu)C^2}
 \ge
 \frac{1}{12\varepsilon_0(c+\mu)C^2}
 =
 \frac{13}{12}L_0
 >
 L_0.
\]
Therefore, $\ell_d(\widetilde\alpha[u,v]) \leq \lambda|u-v| + 1$ holds whenever $|u-v| \leq L_0$. 

Proposition~\ref{prop:curve-approx} then implies $\widetilde\alpha \subset B_d(\gamma, L_0)$. Applying Lemma~\ref{lem:reverse} with $Q=\widetilde{\alpha}$ gives $\gamma \subset B_d(\widetilde\alpha, 2cL_0+\mu)$.
Then Theorem~\ref{thm:weighted-close} yields that there exists $K=K(\delta,C,c,\mu)\geq1$ such that
\[
\ell_\rho(\gamma)
   \le K\ell_\rho(\widetilde{\alpha})
   \le K\ell_\rho(\alpha).
\]
This completes the proof of Theorem~\textup{\ref{thm:G-H}}.
\end{proof}

\section{Uniformity of quasigeodesics in conformally deformed spaces}\label{sec:uniformity of quasigeodesic}
In this section, we will prove that $(c,\mu)$-quasigeodesics in complete Gromov hyperbolic spaces are uniform curves in the conformally deformed spaces $X_{\varepsilon}= (X,d_\varepsilon)$ when $\varepsilon$ is sufficiently small, quantitatively.

\begin{proof}[Proof of Theorem~\textup{\ref{thm:uniform}}]
	First of all, by \cite[Proof of Theorem 1.2]{AlluJose}, each $X_\varepsilon = (X,d_\varepsilon)$ is a noncomplete bounded rectifiably connected metric space. Thus, it suffices to prove that there exist constants $\varepsilon_0=\varepsilon_0(\delta,c,\mu)$ and $A=A(\delta,c,\mu)$
  such that every $(c,\mu)$-quasigeodesic in $(X,d)$ is an $A$-uniform curve in $X_\varepsilon$ for $0 < \varepsilon \leq \varepsilon_0$. 
	
	Notice that by Theorem~\ref{thm:G-H} with $\rho =\rho_\varepsilon$ and $C=1$, there exist constants $\varepsilon_0 = \varepsilon_0(\delta, c, \mu) > 0$ and $A = A(\delta, c,\mu) \geq 1$ so that if $0 < \varepsilon \leq \varepsilon_0$, then for each $\gamma \in \Lambda_{x,y}^{c,\mu}(X,d)$, it holds 
	$$
		\ell_\varepsilon(\gamma) \leq A d_\varepsilon(x,y).
	$$
 It follows that each  $\gamma \in \Lambda_{x,y}^{c,\mu}(X,d)$ is $A$-quasiconvex in $X_\varepsilon$.

	Next, we show that every $(c,\mu)$-quasigeodesic  $\gamma \in \Lambda_{x,y}^{c,\mu}(X,d)$ satisfies the $A$-double cone condition.
	To this end, let $\alpha: x \curvearrowright y$ be a $2$-short arc. 
	For $z \in \gamma$, by Lemma~\ref{quasigeodesic ieq}, we can choose $a=x$ or $y$ so that when $u \in \gamma[a,z]$, 
	\begin{equation}\label{4.1}
		|p-u| \geq |p-z| + |u-z| - \kappa.
	\end{equation}
where $\kappa=\kappa(\delta,c,\mu) = 8\delta + 6M_0 + 2\delta_0 + 20$.
	
	Now by \eqref{4.1}, we may estimate the conformal length as follows:
	\beq\label{4.2}
	\begin{aligned}
		\ell_\varepsilon(\gamma[a,z]) &= \int_{\gamma[a,z]} \exp(-\varepsilon|p-u|) \, |du| \\
		&\leq \exp(\varepsilon \kappa) \exp(-\varepsilon|p-z|) \int_{\gamma[a,z]} \exp(-\varepsilon|u-z|) \, |du|\\
		&\leq \exp(\varepsilon \kappa) \rho_\varepsilon(z) \int_0^\infty \exp\!\Big(-\frac{\varepsilon (t-\mu)}{c}\Big) \, dt \\
		&= \frac{c \, \exp\left(\varepsilon \left(\kappa+\mu/c\right)\right)}{\varepsilon} \rho_\varepsilon(z)\leq \frac{c \, \exp\left(\varepsilon_0 \left(\kappa+\mu/c\right)\right)}{\varepsilon} \rho_\varepsilon(z)
	\end{aligned}
	\eeq
	
	Next, we estimate $d_\varepsilon(z)$ for $z\in X$ and $0 < \varepsilon \leq \varepsilon_0$. 
	
	
	For any $\zeta\in \partial X_\varepsilon$, Proposition~\ref{prop:boundary-id} implies that there exists $\{x_j\}$ and $j_0$ such that $|z-x_j| > 1/\varepsilon$ if $j\geq j_0$ and $x_j\stackrel{d_\varepsilon}{\longrightarrow}\zeta,$ when $j\rightarrow \infty$. 
	Let $\beta_j$ be a rectifiable curve joining $z$ and $x_j$ starting from $z$ and select a subcurve $\beta_j'$ of $\beta_j$ with $\ell_d(\beta_j') = 1/\varepsilon$. Then by~\eqref{eq:conformal-metric} and \eqref{eq:inequality}, we have
	\beqq
	\begin{aligned}
		d_\varepsilon(z,x_j) &= \inf_{\beta_j} \int_{\beta_j} \rho_\varepsilon \, ds \geq \inf_{\beta_j} \int_{\beta_j'} \rho_\varepsilon \, ds \\
		&\geq \int_0^{1/\varepsilon} \rho_\varepsilon(z) e^{-\varepsilon t} \, dt \geq \bigl(1 - e^{-1}\bigr) \frac{\rho_\varepsilon(z)}{\varepsilon} \\
		&\geq \frac{\rho_\varepsilon(z)}{e\varepsilon}.
	\end{aligned}
	\eeqq
	In particular, this implies that for all $z \in X$, it holds
	\beq\label{4.3}
	d_\varepsilon(z) = d_\varepsilon(z,  \partial X_\varepsilon) \geq \frac{1}{e \varepsilon } \rho_\varepsilon(z).
	\eeq
	
	Combining \eqref{4.2} and \eqref{4.3} gives 
	\[
	\ell_\varepsilon(\gamma[a,z]) 
	\leq \exp\left(\varepsilon_0 \left(\kappa+\mu/c \right) + \log c + 1\right) \, d_\varepsilon(z).
	\]
	Without loss of generality, we may assume $A\geq \exp\left(\varepsilon_0 \left(\kappa+\mu/c \right) + \log c + 1\right)$.
	Consequently, for all $z\in \gamma,$ it holds 
$$\ell_\varepsilon(\gamma[x,z])\wedge \ell_\varepsilon(\gamma[y,z])\leq A d_\varepsilon(z).$$
This establishes the desired $A$-double cone condition. Therefore every $(c,\mu)$-quasigeodesic in $(X,d)$ is an $A$-uniform curve in $X_\varepsilon$ for $0 < \varepsilon \leq \varepsilon_0$.

In the above proof, taking $c=\mu=1$ in particular yields the second desired conclusion and thus completes the proof of Theorem~\ref{thm:uniform}.

\end{proof}

\section{Uniform spaces are Gromov hyperbolic in the quasihyperbolic metric}\label{sec:uniform spaces are Gromov hyperbolic}

In this section, we prove Theorem~\ref{thm:Gromov and roughly starlike}. 

\subsection{Quasigeodesics in inner uniform domains}
\begin{thm}
	\label{thm:intrinsic-210}
	Let $(\Omega,d)$ be an $A$-(inner) uniform space and $k$ be its
	quasihyperbolic metric. For every $c\ge1$ and $\mu>0$, there exists
	$B=B(A,c,\mu)\ge1$ such that for all points $y_1,y_2\in \Omega$, each curve
	$\gamma\in\Lambda^{c,\mu}_{y_1,y_2}(\Omega,k)$
	is a $B$-(inner) uniform curve in $(\Omega,d)$. 
\end{thm}

\begin{proof}
	We only consider the case for uniform spaces, as the proof of the other case  is similar.
	Fix $\gamma\in\Lambda^{c,\mu}_{y_1,y_2}(\Omega,k)$. We divide the proof into three main steps.
	\smallskip
	
	\textbf{Step 1.} Construct a dyadic decomposition of $\gamma$.
	\smallskip

	After reparameterizing $\gamma$ by $d$-arc-length from $y_1$, write
	$$\gamma:[0,L_1]\longrightarrow\Omega,
	\quad L_1=\ell_d(\gamma).$$
	Set 
	$D=\max\limits_{0\le t\le L_1}d_\Omega(\gamma(t)).$
	For $i=1,2$, let
	$N_i$ be the unique nonnegative integer such that
	\begin{equation*}
		\frac{D}{2^{N_i+1}}
		<
		d_\Omega(y_i)
		\le
		\frac{D}{2^{N_i}}.
	\end{equation*}
	For $k=0,\cdots,N_1$, let
	\[
	t_k^1
	=
	\min\left\{
	t\in[0,L_1]:d_\Omega(\gamma(t))=\frac{D}{2^k}
	\right\},
	\quad
	x_k^1=\gamma(t_k^1),
	\]
	and for $l=0,\cdots,N_2$, let
	\[
	t_l^2
	=
	\max\left\{
	t\in[0,L_1]:d_\Omega(\gamma(t))=\frac{D}{2^l}
	\right\},
	\quad
	x_l^2=\gamma(t_l^2).
	\]
	These parameters exist by continuity and satisfy
	\[
	0\le t_{N_1}^1<\cdots< t_1^1< t_0^1
	\le t_0^2< t_1^2<\cdots< t_{N_2}^2\le L_1.
	\]
	We now define the nonoverlapping subcurves $\gamma_\nu$, modulo common
	endpoints, by \,$\gamma_{-(N_1+1)}=\gamma|_{[0,t_{N_1}^1]}, \,              \gamma_{-k}=\gamma|_{[t_k^1,t_{k-1}^1]}$ for $k=1,\cdots,N_1$,\,     $\gamma_0=\gamma|_{[t_0^1,t_0^2]},\,                                    \gamma_l=\gamma|_{[t_{l-1}^2,t_l^2]}$ for $l=1,\cdots,N_2$, and           $\gamma_{N_2+1}=\gamma|_{[t_{N_2}^2,L_1]}.$  
	If one of the defining parameters is an endpoint, the corresponding terminal
	subcurve may be degenerate. The family $\{\gamma_\nu:-N_1-1\le\nu\le N_2+1\}$
	contains exactly $N_1+N_2+3$ pieces and covers $\gamma$.
	
	Let $a_\nu,b_\nu$ be the endpoints of $\gamma_\nu$. We have
	\begin{equation}
		\label{eq:dyadic-pointwise}
		d_\Omega(z)
		\le
		\frac{D}{2^{|\nu|-1}},
		\quad z\in\gamma_\nu,
	\end{equation}
	and
	\begin{equation}\label{eq:dyadic-endpoints}
		d_\Omega(a_\nu)
		\wedge
		d_\Omega(b_\nu)
		\ge
		\frac{D}{2^{|\nu|}}.
	\end{equation}
	For $\nu=0$, the first bound simply reads $d_\Omega(z)\le2D$;
	in fact the sharper bound $d_\Omega(z)\le D$ holds.
	\smallskip
	
	\textbf{Step 2.} Verify the double cone condition.
	\smallskip
	
	On each dyadic piece, by \eqref{eq:dyadic-pointwise}, we have 
	\begin{equation*}
		\ell_k(\gamma_\nu)
		=
		\int_{\gamma_\nu}\frac{ds}{d_\Omega(z)}
		\ge
		\frac{2^{|\nu|-1}}{D}\ell_d(\gamma_\nu).
	\end{equation*}
	On the other hand, Lemma~\ref{lem:intrinsic-213}, \eqref{eq:prop of subcurve} and \eqref{eq:dyadic-endpoints} yield
	\begin{equation}
		\label{eq:block-upper-k}
		\begin{aligned}
			\ell_k(\gamma_\nu)\leq c\,k(a_\nu,b_\nu)+\mu
			&\le
			4cA^2
			\log\left(
			1+
			\frac{2^{|\nu|}\ell_d(\gamma_\nu)}{D}
			\right)
			+\mu.
		\end{aligned}
	\end{equation}
	Combining the above two estimates gives
	\begin{equation*}
		\frac{2^{|\nu|-1}\ell_d(\gamma_\nu)}{D}
		\le
		4cA^2\log\left(1+\frac{2^{|\nu|}\ell_d(\gamma_\nu)}{D}\right)+\mu.
	\end{equation*}
 Since $\log(1+s)\le\sqrt{s}$ for $s\geq 0$, we obtain from the above that 
	\begin{equation*}
		\frac{2^{|\nu|}\ell_d(\gamma_\nu)}{D}
		\le
		T
		:=
		\left(
		4cA^2+
		\sqrt{16c^2A^4+2\mu}
		\right)^2.
	\end{equation*}
	Thus
	\begin{equation}
		\label{eq:block-length}
		\ell_d(\gamma_\nu)
		\le
		TD\,2^{-|\nu|},
	\end{equation}
	and so 
	\begin{equation*}
		\ell_k(\gamma_\nu)
		\stackrel{\eqref{eq:block-upper-k}}{\le} K_1:=4cA^2\log(1+T)+\mu.
	\end{equation*}
	For every $z\in\gamma_\nu$, it holds 
	\[
	\left|
	\log\frac{d_\Omega(z)}{d_\Omega(a_\nu)}
	\right|
	\stackrel{\eqref{eq:quasihyperbolic inequality}}{\le}
	k(z,a_\nu)
	\le
	\ell_k(\gamma_\nu)
	\le K_1.
	\]
	Combining this with \eqref{eq:dyadic-endpoints}, we have
	\begin{equation}
		\label{eq:block-boundary-lower}
		d_\Omega(z)
		\ge
		e^{-K_1}D\,2^{-|\nu|},
		\quad z\in\gamma_\nu.
	\end{equation}
	
	If $z\in\gamma_\nu$ and $\nu\le0$, then the subcurve from $y_1$ to $z$ is
	contained in the union of the blocks on the $y_1$-side having dyadic
	indices at least $|\nu|$. Therefore, by \eqref{eq:block-length},
	\[
	\ell_d(\gamma[y_1,z])
	\le
	TD\sum_{j\ge|\nu|}2^{-j}
	\le
	2TD\,2^{-|\nu|}.
	\]
	If $\nu\ge0$, then an analogous estimate holds for $\ell_d(\gamma[z,y_2])$.
	Combining these estimates with \eqref{eq:block-boundary-lower}, we obtain
	\begin{equation}
		\label{eq:double-cone-explicit}
		\ell_d(\gamma[y_1,z])
		\wedge
		\ell_d(\gamma[z,y_2])
		\le
		B_1d_\Omega(z),
		\quad
		B_1=2Te^{K_1}.
	\end{equation}
		\smallskip
	
	\textbf{Step 3.} Verify the quasiconvexity condition.
	\smallskip
	
	
	Choose $y_1',y_2'\in\gamma$ so that
	\[
	\ell_d(\gamma[y_1,y_1'])
	=
	\ell_d(\gamma[y_2',y_2])
	=
	\frac{1}{2}d(y_1,y_2).
	\]
	Then 
	\[
	d(y_1',y_2')
	\le
	d(y_1',y_1)+d(y_1,y_2)+d(y_2,y_2')
	\le
	2d(y_1,y_2),
	\]
	and by \eqref{eq:double-cone-explicit},
	\begin{equation*}
		d_\Omega(y_1')
		\wedge
		d_\Omega(y_2')
		\ge
		\frac{1}{2B_1}d(y_1,y_2).
	\end{equation*}
	These imply
	\[
	\log\left(1+\frac{\ell_d(\gamma[y_1',y_2'])}{ d_\Omega(y_1')    \wedge   d_\Omega(y_2')}\right)
	\stackrel{\eqref{eq:quasihyperbolic length inequality}}{\le}
	\ell_k(\gamma[y_1',y_2'])\stackrel{\eqref{eq:prop of subcurve}}{\leq} c\,k(y_1',y_2')+\mu\stackrel{\eqref{eq:j-k-upper}}{\leq} 	4cA^2\log(1+4B_1)+\mu,
	\] 
	from which we conclude 
	\begin{equation*}
		\ell_d(\gamma[y_1',y_2'])
		\le
		\bigl(e^{4cA^2\log(1+4B_1)+\mu}-1\bigr)d_\Omega(y_1')    \wedge   d_\Omega(y_2').
	\end{equation*}
	
	Set
	\begin{equation}\label{eq:explicit-R}
		R_1=R_1(A,c,\mu)
		:=
		\frac12
		+
		\frac{1}{
			\exp\!\left(\dfrac{1}{4A^2(c+\mu)}\right)-1
		}.
	\end{equation}
	Since $A\ge1$, $c\ge1$, and $\mu>0$, we have $R_1>1$.
	
	If $d_\Omega(y_1')    \wedge   d_\Omega(y_2')\le R_1d(y_1, y_2)$, then
	\[
	\ell_d(\gamma)
	=
	d(y_1, y_2)+\ell_d(\gamma[y_1',y_2'])
	\le
	\left[1+R_1\bigl(e^{4cA^2\log(1+4B_1)+\mu}-1\bigr)\right]d(y_1, y_2).
	\]
	
If $d_\Omega(y_1')    \wedge   d_\Omega(y_2')>R_1d(y_1, y_2)$, then
	\[
	d_\Omega(y_1)
	\wedge
	d_\Omega(y_2)\geq d_\Omega(y'_1)
	\wedge
	d_\Omega(y'_2)-\frac{1}{2}d(y_1,y_2)
	>
	\left(R_1-\frac12\right)d(y_1, y_2),
	\]
	and so  
	\[
\ell_k(\gamma)\stackrel{\eqref{eq:prop of curve}}{\leq}(c+\mu)k(y_1,y_2)	\stackrel{\eqref{eq:j-k-upper}}{\le}4 (c+\mu)
	A^2\log\left(1+\frac{1}{R_1-1/2}\right)\stackrel{\eqref{eq:explicit-R}}{\leq} 1.
	\]
	Moreover, from \eqref{eq:quasihyperbolic length inequality} and \eqref{eq:j-k-upper}, we obtain
	\beqq
	\begin{aligned}
		\frac{\ell_d(\gamma)}{d_\Omega(y_1)\wedge d_\Omega(y_2)}
		&\le
		e^{\ell_k(\gamma)}-1\le2\ell_k(\gamma)\le2(c+\mu)k(y_1,y_2)\\
		&\le
		8A^2(c+\mu)\frac{d(y_1, y_2)}{d_\Omega(y_1)\wedge d_\Omega(y_2)}.
	\end{aligned}
	\eeqq
	Thus
	\[
	\ell_d(\gamma)\le8A^2(c+\mu)d(y_1, y_2).
	\]
	In both cases, $\gamma$ is $	\max\left\{
	1+R_1\bigl(e^{4cA^2\log(1+4B_1)+\mu}-1\bigr),
	8A^2(c+\mu)
	\right\}$-quasiconvex.
	
	Therefore, every curve
	$\gamma\in\Lambda^{c,\mu}_{y_1,y_2}(\Omega,k)$
	is a $B$-uniform curve with
	\[
	B=B(A,c,\mu)=
	\max\left\{
	B_1,
	1+R_1\bigl(e^{4cA^2\log(1+4B_1)+\mu}-1\bigr),
	8A^2(c+\mu)
	\right\}.
	\]
	The proof is complete. 
\end{proof}
For convenience, from now on, we set $B_0=B(A,1,\mu)$ and $B'_0=B(A,1,1)$.


To prove that every inner uniform space is Gromov hyperbolic, by Lemma \ref{lem:vaisala-package}, it suffices to show that every $(c,\mu)$-quasigeodesic triangle is $\vartheta$-thin in $(\Omega,k)$. This criterion is standard in the literature (see \cite{BHK}); once established, the Gromov hyperbolicity of inner uniform spaces follows immediately.

\begin{thm}\label{thm:intrinsic-36}
	Let $(\Omega,d)$ be an $A$-inner uniform space 
	and $k$ the associated quasihyperbolic metric. Then the following statements hold.
	\begin{enumerate}
		\item\label{thm:result1} For all $c\geq1$, $\mu>0$, there exists $\vartheta=\vartheta(A,c,\mu)$ so that every triangle whose three sides are $(c,\mu)$-quasigeodesics  is $\vartheta$-thin in $(\Omega,k)$.
		
		\item\label{thm:result2} If $\alpha:x\curvearrowright y$ is a $(1,\mu)$-quasigeodesic and $p\in\Omega$, then
		\begin{equation}\label{eq:gp-curve-comparison}
			k(p,\alpha)-2\vartheta_0-\mu
			\le
			(x|y)_p
			\le
			k(p,\alpha)+\frac{\mu}{2},
		\end{equation}
		\item\label{thm:result3} $(\Omega,k)$ is $\delta_1$-Gromov hyperbolic, where $\delta_1=\delta_1(A)= 3\vartheta'_0+\frac{3}{2}$.
	\end{enumerate}
	Here $\vartheta_0=\vartheta(A,1,\mu)$ and $\vartheta'_0=\vartheta(A,1,1)$.
\end{thm}

\begin{proof}
	Let $(\Omega,d)$ be $A$-inner uniform space 
	and let $\alpha:x\curvearrowright y,\,\beta:y\curvearrowright z,\,\gamma:z\curvearrowright x$
	be $(c,\mu)$-quasigeodesics in $(\Omega,k)$. By Theorem~\ref{thm:intrinsic-210}, all three are
	$B$-inner uniform curves in $(\Omega,d)$, where $B=B(A,c,\mu)$. Fix $u\in\alpha$, and assume
	without loss of generality that
	\[
	\ell_d(\alpha[x,u])
	\wedge
	\ell_d(\alpha[u,y])= \ell_d(\alpha[x,u]).
	\]
	Then
	\begin{equation*}
		d_\Omega(u)
		\ge
		\frac{1 }{B}\ell_d(\alpha[x,u]).
	\end{equation*}
	Moreover,
	\[
	\frac{2}{B} \ell_d(\alpha[x,u])
	\le\frac{ 1}{B}\ell_d(\alpha)\leq \sigma_\Omega(x,y)
	\le \sigma_\Omega(x,z)+\sigma_\Omega(z,y)\le\ell_d(\gamma)+\ell_d(\beta).
	\]
	
	If $\ell_d(\gamma)< \frac{1}{B} \ell_d(\alpha[x,u])$, then $\ell_d(\beta)> \frac{1}{B} \ell_d(\alpha[x,u])$. Choose $v\in\beta$ so that
	\[
	\ell_d(\beta[z,v])=\frac{1}{2B} \ell_d(\alpha[x,u]),
	\]
	which implies $\ell_d(\beta[v,y])\ge  \frac{1}{2B} \ell_d(\alpha[x,u])$.
	Then 
	$$ \sigma_\Omega(u,v)\leq\ell_d(\alpha[x,u])+\ell_d(\gamma)+\ell_d(\beta[z,v])
	\le
	\frac{2B+3}{2B} \ell_d(\alpha[x,u]),$$
	and 
	\[
	d_\Omega(v)
	\ge
	\frac{ 1}{2B^2}\ell_d(\alpha[x,u]).
	\]
	
	If $\ell_d(\gamma)\ge  \frac{1}{B} \ell_d(\alpha[x,u])$, choose $v\in\gamma$ so that
	\[
	\ell_d(\gamma[x,v])=\frac{1}{2B} \ell_d(\alpha[x,u]).
	\]
	Similarly, the same estimates hold, with a smaller bound for $\sigma_\Omega(u,v)$. 
	
	Thus, in both
	cases,
	\[
	d_\Omega(u)\wedge d_\Omega(v)
	\ge
	\frac{1}{2B^2}\ell_d(\alpha[x,u]),
	\]
	and
	\[
	\sigma_\Omega(u,v)
	\le
	\frac{2B+3}{2B} \ell_d(\alpha[x,u]).
	\]
	By \eqref{eq:inner upper},
	\[
	k(u,v)
	\le
	4A^2\log\bigl(1+B(2B+3)\bigr).
	\]
	Hence
	\[
	k(u,\beta\cup\gamma)\le \vartheta,
	\]
	where $\vartheta=\vartheta(A,c,\mu)$.
	The same argument for the other two sides shows that the triangle is $\vartheta$-thin. This proves statement \eqref{thm:result1}.
	
	Lemma~\ref{lem:vaisala-package}, applied with $\rho=k$, gives
	\eqref{eq:gp-curve-comparison}, while applied with $\rho=k$, $c=1$ and $\mu=1$, shows that $(\Omega,k)$ is $\delta_1$-Gromov hyperbolic, where $\delta_1=3\vartheta'_0+\frac{3}{2}$ and $\vartheta'_0=\vartheta_0(A,1,1).$
	Therefore, statements \eqref{thm:result2} and \eqref{thm:result3} hold.
\end{proof}
	
\subsection{Roughly starlikeness of bounded uniform spaces}	
As for the boundary correspondence, we obtain the following bijection between the boundary of a uniform space and its Gromov boundary, which parallels the boundary identification in \cite[Section 3]{BHK}. 
\begin{prop}\label{prop:boundary}
	Let $(\Omega,d)$ be an $A$-uniform space and 
	$k$ the associated quasihyperbolic metric. Then the map
	\[
	\Psi:\partial_G(\Omega,k)\longrightarrow\partial^{*}\Omega=\partial\Omega\cup\{\infty\},
	\quad
	\Psi([\{x_n\}])=\lim_{n\to\infty}x_n,
	\]
	is a bijection. In particular, if $(\Omega,d)$ is bounded, then $\partial^{*}\Omega=\partial\Omega$.
\end{prop}	
\begin{proof}
	We divide the proof into three main steps.
	\smallskip
	
	\textbf{Step 1.} $\Psi$ is well-defined. 
	\smallskip 
	
	Let $\{x_n\}$ be a Gromov sequence in $(\Omega,k)$. We claim that if $\{x_n\}$ is bounded, then it is a $d$-Cauchy sequence and $x_n\xrightarrow{n\to\infty} a\in \partial\Omega$.
	
	Suppose the claim fails. Then there exist $\varsigma>0$ and
	sequences of indices $n_j,m_j\to\infty$ such that $d(x_{n_j},x_{m_j})\ge \varsigma.$
	Fix $w\in\Omega$ so that $d(w, x_n)\leq C'$. For each $j\in \mathbb{N}$, choose a $(1,1)$-quasigeodesic
	$\gamma_j:x_{n_j}\curvearrowright x_{m_j}$ in $(\Omega,k)$.
	Then $\gamma_j$ is a $B'_0$-uniform curve by Theorem~\ref{thm:intrinsic-210}.
	Hence
	$$\ell_d(\gamma_j)\leq B'_0\,d(x_{n_j}, x_{m_j})\leq 2B'_0C'.$$
	Let $z_j\in \gamma_j$ bisect $\ell_d(\gamma_j)$. Then 
	\[
	d_\Omega(z_j)
	\ge
	\frac{\ell_d(\gamma_j)}{2B'_0}
	\ge
	\frac{\varsigma}{2B'_0}
	\]
	and
	\[
	k(w,z_j)\stackrel{\eqref{eq:j-k-upper}}{\leq}4A^2\log\left(1+
	\frac{d(w,z_j)}
	{d_\Omega(w)\wedge d_\Omega(z_j)}\right)
	\le4A^2\log\left(1+\frac{2B'_0C'}
	{d_\Omega(w)\wedge\frac{\varsigma}{2B'_0}}
	\right)
	=:C_\varsigma.
	\]
	This implies that $k(w,\gamma_j)\le C_\varsigma$. Moreover, 
	the upper bound in \eqref{eq:gp-curve-comparison} yields
	\[
	(x_{n_j}|x_{m_j})_w
	\le
	C_\varsigma+\frac12,
	\]
	contradicting with the definition of a Gromov sequence. Thus $\{x_n\}$ is
	$d$-Cauchy sequence.
	Since $k(w,x_n)=(x_n|x_n)_w\xrightarrow{n\to\infty}\infty$ and 
	$$k(w,x_n)\leq 4A^2\log \left(1+\frac{d(w,x_n)}{d_{\Omega}(w)\wedge d_{\Omega}(x_n)}\right)\leq 4A^2\log \left(1+\frac{C'}{d_{\Omega}(w)\wedge d_{\Omega}(x_n)}\right),$$
	we infer that $d_{\Omega}(x_n)\xrightarrow{n\to\infty}0$. 
	Hence $x_n\xrightarrow{n\to\infty} a\in \partial\Omega$. This proves the claim. Consequently, if $(\Omega,d)$ is bounded, then $\partial^{*}\Omega=\partial\Omega$.
	
	Next suppose that $\{x_n\}$ and $\{y_n\}$ are equivalent Gromov sequences. Then 
	$\{x_n\}$ and $\{y_n\}$ are either both  bounded or both  unbounded.
	Indeed, if not, we may assume without loss of generality that $\{x_n\}$ is bounded and $\{y_n\}$ is unbounded. 
	
	For each $n\in \mathbb{N}$, choose a $(1,1)$-quasigeodesic
	$\gamma_n:x_{n}\curvearrowright y_{n}$ in $(\Omega,k)$.
	Then $\gamma_n$ is a $B'_0$-uniform curve
	and
	$\ell_d(\gamma_n)\geq d(x_n,y_n)\xrightarrow{n\to\infty}\infty.$
	For sufficiently large n, select $z_n\in \gamma_n$ so that $\ell_d(\gamma_n[x_n,z_n])=1$. Then
	$d_\Omega(z_n)\ge
	1/B'_0$ and	by \eqref{eq:j-k-upper}, it holds 
	$$k(w,z_{n})\leq 4A^2\log\left(1+\frac{d(w,z_{n})}{d_\Omega(w)\wedge d_\Omega(z_{n})}\right)\leq 4A^2\log\left(1+\frac{d(w,x_{n})+1}{d_\Omega(w)\wedge\frac{1}{B'_0}}\right).$$
	Hence $k(w,\gamma_{n})$ is bounded. Combining this with  \eqref{eq:gp-curve-comparison}, we obtain that $(x_{n}|y_{n})_w$ is bounded, which contradicts with $(x_{n}|y_{n})_w\to\infty$ as $n\to\infty.$

	If $\{x_n\}$ and $\{y_n\}$ are unbounded, then $x_n, y_n\xrightarrow{n\to\infty} \infty$. Next, suppose that $\{x_n\}$ and $\{y_n\}$ are bounded.
	Assume $x_n\to a,\,y_n\to b$ in the $d$-completion as $n\to\infty$. If $a\ne b$, then for all sufficiently large $n$, we have 
	\[
	d(x_n,y_n)\ge\frac12d(a,b).
	\]
	Applying the preceding midpoint argument to a $(1,1)$-quasigeodesic joining
	$x_n$ to $y_n$ in $(\Omega,k)$, we obtain a uniform upper bound for
	$(x_n|y_n)_w$ similar to the one above, contradicting again with $(x_{n}|y_{n})_w\to\infty$ as $n\to\infty$. Hence $a=b$, and $\Psi$ is
	well-defined.
	\smallskip
	
	\textbf{Step 2.} $\Psi$ is injective.
	\smallskip 
	
	Let $\xi=[\{x_n\}],\,\eta=[\{y_n\}]
	\in\partial_G(\Omega,k)$
	and suppose that $\Psi(\xi)=\Psi(\eta)=a\in\partial^{*}\Omega.$
	Thus $ x_n\to a,\, y_n\to a$ in the $d$-completion as $n\to\infty.$ Our aim is to show that $\xi=\eta$. 
	\smallskip 
	
	\textbf{Case 1}: $a=\infty$.
	\smallskip 
	
	For each $n,m\in \mathbb{N}$, choose $(1,1)$-quasigeodesics
	$\gamma_{n}:w\curvearrowright x_n$ and $\gamma'_{m}:w\curvearrowright y_m$ in $(\Omega,k)$.
	Then $\gamma_{n}$ and $\gamma'_{m}$ are $B'_0$-uniform curve in $(\Omega,d)$. Parameterize $\gamma_{n}$ and $\gamma'_{m}$ by $d$-arc-length  from $w$. For $s=\frac{1}{2}\left(\ell_d(\gamma_n)\wedge \ell_d(\gamma'_m)\right)$, 
	$d_{\Omega}(\gamma_{n}(s))\geq s/B'_0$ and $d_{\Omega}(\gamma'_{m}(s))\geq s/B'_0.$  
	Then by \eqref{eq:quasihyperbolic inequality},
	$$k(w,\gamma_{n}(s))\geq\log\frac{d_{\Omega}\left(\gamma_{n}(s)\right)}{d_{\Omega}(w)}\geq\log\frac{s}{B'_0d_{\Omega}(w)}, $$
	and 
	$$k(w,\gamma'_{m}(s))\geq\log\frac{d_{\Omega}\left(\gamma'_{m}(s)\right)}{d_{\Omega}(w)}\geq\log\frac{s}{B'_0d_{\Omega}(w)}. $$
	It follows from \eqref{eq:j-k-upper} that
	$$k(\gamma_{n}(s), \gamma'_{m}(s))\leq 4A^2\log\left(1+\frac{d\left(\gamma_{n}(s), \gamma'_{m}(s)\right)}{d_{\Omega}\left(\gamma_{n}(s)\right)\wedge d_{\Omega}\left(\gamma'_{m}(s)\right)}\right)\leq 4A^2\log(1+2B'_0).$$
	In addition, since $\gamma_n$ and $\gamma'_m$ are $(1,1)$-quasigeodesics in $(\Omega,k)$, we have
	$$k(w,x_n)\geq k(w,\gamma_{n}(s))+k(\gamma_{n}(s),x_n)-1,$$
	$$k(w,y_m)\geq k(w,\gamma'_{m}(s))+k(\gamma'_{m}(s),y_m)-1.$$
	Combining these inequalities with triangle's inequality, we obtain
	\beqq
	\begin{aligned}
		(x_n|y_m)_w&=\frac{1}{2}\left(k(w,x_n)+k(w,y_m)-k(x_n,y_m)\right)\\
		&\geq  \frac{1}{2}\left(k(w,\gamma_{n}(s))+k(w,\gamma'_{m}(s))-k\left(\gamma_{n}(s),\gamma'_{m}(s)\right)\right)-1\\
		&\geq\log\frac{s}{B'_0d_{\Omega}(w)}- 2A^2\log(1+2B'_0)-1.
	\end{aligned}
	\eeqq
	Since $s=\frac{1}{2}\left(\ell_d(\gamma_n)\wedge \ell_d(\gamma'_m)\right)\xrightarrow{n,m\to\infty}\infty$, $(x_n|y_m)_w \xrightarrow{n,m\to\infty}\infty$, which shows $\{x_n\}$ and $\{y_n\}$ are equivalent Gromov sequences. Thus $\xi=\eta$.
	\smallskip 
	
	\textbf{Case 2}: $a\in\partial\Omega$.
	\smallskip 
	
	For $n,m\ge1$, choose a $(1,1)$-quasigeodesic
	$\gamma_{n,m}:x_n\curvearrowright y_m$ in $(\Omega,k)$.
	Then $\gamma_{n,m}$ is a $B'_0$-uniform curve in $(\Omega,d)$. Hence
	\[
	\ell_d(\gamma_{n,m})
	\le B'_0\,d(x_n,y_m)\xrightarrow{n,m\to\infty}0.
	\]
	For every $z\in\gamma_{n,m}$, 
	we have
	\[
	\begin{aligned}
		d_\Omega(z)
		&\le d(z,a)\le d(z,x_n)+d(x_n,a)\le \ell_d(\gamma_{n,m})+d(x_n,a),
		\end{aligned}
	\]
	which implies
	$\sup\limits_{z\in\gamma_{n,m}}d_\Omega(z)\xrightarrow{n,m\to\infty}0$.
	It follows that
	\[
	k(w,\gamma_{nm})
	\ge
	\log\frac{d_\Omega(w)}
	{\sup_{z\in\gamma_{n,m}}d_\Omega(z)}
	\xrightarrow{n,m\to\infty}\infty.
	\]
	In this case, we infer from \eqref{eq:gp-curve-comparison} that $(x_n|y_m)_w\xrightarrow{n,m\to\infty}\infty$.	Thus $\xi=\eta$.
	
	This proves that $\Psi$ is injective.
	\smallskip
	
	\textbf{Step 3.} $\Psi$ is surjective.
	\smallskip 
	
	For each $a\in\partial^{*}\Omega$, we shall prove that there is a Gromov sequence  $\{x_n\}$ so that $\Psi([\{x_n\}])=a$.
	
	If $a=\infty$, then $(\Omega,d)$ is unbounded and there exists an unbounded sequence $\{x_n\}$ so that $x_n\to\infty$ as $n\to\infty$. By the preceding discussion, we have $
	(x_n|x_m)_w\xrightarrow{n,m\to\infty}\infty$. Hence $\{x_n\}$ is a Gromov sequence and $\Psi([\{x_n\}])=\infty=a$.
	
	If $a\in\partial\Omega$, then choose $x_n\in\Omega$ such
	that $x_n\xrightarrow{n\to\infty} a$ with respect to $d$. For $n,m\ge1$, let
	$\gamma_{nm}:x_n\curvearrowright x_m$
	be a $(1,1)$-quasigeodesic in $(\Omega,k)$. Then $\gamma_{nm}$ is $B'_0$-quasiconvex in $(\Omega,d)$ and so
	\[
	\ell_d(\gamma_{nm})
	\le B'_0\,d(x_n,x_m)
	\xrightarrow{n,m\to\infty}0.
	\]
	For every $z\in\gamma_{nm}$, we have
	$$
	d_\Omega(z)
	\le d(z,a)\\
	\le d(z,x_n)+d(x_n,a)\\
	\le \ell_d(\gamma_{nm})+d(x_n,a),
	$$
	which implies $\sup\limits_{z\in\gamma_{nm}}d_\Omega(z)\xrightarrow{n,m\to\infty}0$. Consequently, $k(w,\gamma_{nm})\xrightarrow{n,m\to\infty}\infty$.
	The lower bound in \eqref{eq:gp-curve-comparison} gives
	$
	(x_n|x_m)_w\xrightarrow{n,m\to\infty}\infty.
	$
	Hence $\{x_n\}$ is a Gromov sequence and $\Psi([\{x_n\}])=a$. 
	
	This proves surjectivity and thus completes the proof.	
\end{proof}

With this proposition at hand, we next show that a bounded $A$-uniform space is roughly starlike.

\begin{thm}\label{thm:roughly starlike}
	Let $(\Omega,d)$ be a bounded $A$-uniform space and $k$ be the quasihyperbolic metric of $\Omega$. Choose $w\in\Omega$ such that
	\begin{equation}\label{eq:deep-basepoint-road}
		d_\Omega(w)
		\ge
		\frac12\sup_{z\in\Omega}d_\Omega(z).
	\end{equation}
	Then for each $h>0$,
	$(\Omega,k)$ is $(H,\nu,h)$-roughly starlike with respect to $w$, where $$H=H(A,h)=4A^2\log\left(1+18B(A,1,h)\right)$$ and $\nu=\nu(A,h)=4\delta_1+2h$, with $\delta_1=\delta_1(A)$ given by Theorem~\textup{\ref{thm:intrinsic-36}}.
\end{thm}
\begin{proof}
	Fix $x\in\Omega$. Our aim is to show that there exists a $(\nu,h)$-road $\bar{\alpha}:w\curvearrowright\xi\in \partial_{G}(\Omega,k)$ such that 
	\[
	k(x,\bar{\alpha})\leq H. 
	\]
	
	Choose $a\in\partial\Omega$ such that
	\begin{equation}\label{eq:approx-nearest-road}
		d(x,a)<2d_\Omega(x).
	\end{equation}
	By Proposition~\ref{prop:boundary}, there exists $\xi\in\partial_G(\Omega,k)$ such that $\Psi(\xi)=a$, where $\xi=[\{u_n\}]$ and $d(u_n,a)\to0$ as $n\to\infty$. 
	
	 By Theorem~\ref{thm:intrinsic-36}, $(\Omega,k)$ is $\delta_1$-Gromov hyperbolic and thus 
	\cite[Theorem 6.7]{Vaisala} yields that, for every $h>0$, there exists a $(\nu,h)$-road
	$\bar{\alpha}:w\curvearrowright\xi,\, \alpha_i:w\curvearrowright u_i$ with $\nu=4\delta_1+2h$.
	The endpoint sequence $\{u_i\}$ represents $\xi$.
	Note that $d(u_i,a)\to0$ and $k(w,u_i)=(u_i|u_i)_w\to\infty$ as $i\to\infty$.
	We can choose $i$ large so that
	\[
	d(u_i,a)<\frac 18d_\Omega(x)
	\quad\text{and}\quad
	k(w,u_i)\ge1.
	\]
	
	Since $\alpha_i$ is $h$-short and $k(w,u_i)\ge1$, 
	it is a $(1,h)$-quasigeodesic in $(\Omega,k)$. Then $\alpha_i$ is a $B'$-uniform curve in $(\Omega,d)$ by Theorem \ref{thm:intrinsic-210}, where $B'=B(A,1,h)$.
	As $d(w,a)\ge d_\Omega(w)\geq\frac{1}{2}d_\Omega(x)$, we have
	\[
	\ell_d(\alpha_i)\ge d(w,u_i)
	\ge d(w,a)-d(a,u_i)
	> \frac{3}{8}d_\Omega(x).
	\]
	Choose $y_i\in\alpha_i$ such that
	$
	\ell_d(\alpha_i[y_i,u_i])=\frac 18d_\Omega(x).
	$
	Then
	$
	\ell_d(\alpha_i[w,y_i])>\frac 14d_\Omega(x)
	$
	and so 
	\[
	\frac{1}{8B'}d_\Omega(x)=\frac{1}{B'}\ell_d\left(\alpha_i[w,y_i]\right)\wedge \ell_d\left(\alpha_i[u_i,y_i]\right) \leq d_\Omega(y_i).
		\]
	On the other hand, by \eqref{eq:approx-nearest-road}, we have 
	\[
	\begin{aligned}
		d(x,y_i)
		\le d(x,a)+d(a,u_i)+d(u_i,y_i)<\frac{9}{4}d_\Omega(x).
	\end{aligned}
	\]
	Consequently, we infer
	\[
	\begin{aligned}
	k(x,\alpha_i)\leq	k(x,y_i)
		\stackrel{\eqref{eq:j-k-upper}}{\le}
		4A^2\log\left(
		1+
		\frac{d(x,y_i)}
		{d_\Omega(x)\wedge d_\Omega(y_i)}\right)\le 4A^2\log(1+18B').
	\end{aligned}
	\]
	It follows that
	\[
	k\bigl(x,\bar{\alpha}\bigr)\le 4A^2\log(1+18B').
	\]
	As $x$ was arbitrary, this implies that $(\Omega,k)$ is $(H, \nu,h)$-roughly starlike with respect to $w$, where $H=H(A,h)=4A^2\log(1+18B')$, $B'=B(A,1,h)$ and $\nu=\nu(A,h)=4\delta_1+2h$.
\end{proof}	

\subsection{Proof of Theorem~\textup{\ref{thm:Gromov and roughly starlike}}}
\begin{proof}[Proof of Theorem~\textup{\ref{thm:Gromov and roughly starlike}}]
Since $(\Omega,d)$ is $A$-quasiconvex,  the identity map $(\Omega,d)\to (\Omega,\ell_d)$ is a homeomorphism. Then \cite[Proposition 2.8]{BHK} implies that $(\Omega,k)$ is complete. That $(\Omega,k)$ is $\delta_1(A)$-Gromov hyperbolic follows from Theorem~\ref{thm:intrinsic-36}. If $(\Omega,d)$ is additionally bounded, then Theorem~\ref{thm:roughly starlike} gives the roughly starlikeness. This completes the proof of Theorem~\ref{thm:Gromov and roughly starlike}.
\end{proof}

\section{Uniformization of non-proper length spaces}\label{sec:final uniformization}
This section is devoted to the proof of Theorem~\ref{thm:one-to-one correspondence}. 
For $0 < \varepsilon \le \varepsilon_1(\delta)$, let $\mathcal{D}$ denote the \emph{uniformization} 
$X \longrightarrow X_\varepsilon$, the letter $\mathcal{D}$ abbreviating \emph{dampening}, and let 
$\mathcal{Q}$ denote the \emph{quasihyperbolization} $(\Omega,d) \longrightarrow (\Omega, k)$. 
Theorem~\ref{thm:uniform} entails that $\mathcal{D}$ sends complete Gromov hyperbolic spaces 
to bounded uniform spaces. Conversely, Theorem~\ref{thm:Gromov and roughly starlike} entails that $\mathcal{Q}$ sends bounded 
uniform spaces to complete roughly starlike Gromov hyperbolic spaces. To establish the asserted one-to-one correspondence, it remains to verify the functoriality 
of both constructions with respect to the appropriate classes of morphisms. 

\subsection{Property of $\mathcal{D}$}
We begin with the following result, which states that the quasisimilarity class of $X_\varepsilon$ depends only on the quasiisometry class of $X$, where $X$ is a complete roughly starlike Gromov hyperbolic space.

\begin{prop}\label{prop:quasiisometric to quasisimilar}
	$\mathcal{D}$ maps mutually quasiisometric complete roughly starlike Gromov hyperbolic spaces to mutually quasisimilar spaces.
\end{prop}

The proof of Proposition~\ref{prop:quasiisometric to quasisimilar} is similar to that used in \cite[Proposition~4.15]{BHK}. Accordingly, we first prove the following lemma.

\begin{lem}\label{lem:varepsilon boundary distance}
	Let $X$ be a complete $\delta$-Gromov hyperbolic and $(H, \nu, h)$-roughly starlike with respect to $w \in X$. Let $\varepsilon_1=\varepsilon_1(\delta)$ be given by Theorem~\textup{\ref{thm:uniform}}. Then for $0<\varepsilon\leq\varepsilon_1$ and $x\in X$, we have
	\beq\label{eq:varepsilon boundary distance}
	\frac{1}{\varepsilon e} \rho_\varepsilon(x) \leq d_\varepsilon(x) \leq A_1\, \frac{\rho_\varepsilon(x)}{\varepsilon},
	\eeq
	where $A_1=A_1(\delta, H, \nu, h)\geq1$.
\end{lem}

\begin{proof}
	The left inequality was already proved in \eqref{4.3} and we only need to prove the right inequality.
	
	For $x\in X$, since $X$ is $(H, \nu, h)$-roughly starlike, there exists  a $(\nu, h)$-road $\bar{\alpha}: w \curvearrowright a \in \partial_G X$ so that $d(x, |\bar{\alpha}|) \leq H$, where $a=\hat{b}=[\{b_n\}]$. Thus there are $h$-short arcs
	$\alpha_i: w \curvearrowright b_i$ and $y_i \in \alpha_i$ such that  $|x - y_i| \leq H+1$. Let $y_j=g_{ij}y_i$ for all $j \geq i$. Then by Definition~\ref{def:road}\,\eqref{road3}, it holds 
	\beq\label{5.10}
	|x - y_j|\leq |x - y_i| + |y_j - y_i| \leq H + \nu+1.
	\eeq
	
	Since $X$ is a length space, for any $\varsigma>0$, we can take $\beta_j: x \curvearrowright y_j$ in $X$ so that 
	$$\ell_d(\beta_j)\leq |x - y_j|+\varsigma\leq H + \nu+\varsigma+1.$$
	By \eqref{5.10}, we have
	$$d_\varepsilon(x, y_j)\leq \int_{\beta_j} \rho_\varepsilon(u) \, ds\leq \rho_\varepsilon(x) \int_0^{H + \nu+\varsigma+1} e^{\varepsilon t}dt\leq\frac{1}{\varepsilon}\rho_\varepsilon(x)(e^{\varepsilon(H+ \nu+\varsigma+1)}-1).$$
	Sending $\varsigma$ to $0$ gives  
	$$d_\varepsilon(x, y_j)\leq\frac{1}{\varepsilon}\rho_\varepsilon(x)(e^{\varepsilon(H+ \nu+1)}-1).$$
	
	Next, we estimate $d_\varepsilon(y_j, b_j)$. 
	Since $g_{ij}$ is a length map, for any $u\in\alpha_j[y_j, b_j]$,
	\beqq
	\begin{aligned}
		|u - w| &\geq \ell_d({\alpha_j[w, y_j]} )+ \ell_d(\alpha_j[y_j, u]) - h \\
		&\geq |w - y_i|+ \ell_d(\alpha_j[y_j, u]) - h\\
		&\geq |x - w| + \ell_d(\alpha_j[y_j, u]) - H- h-1.
	\end{aligned}
	\eeqq
	Then
	\beqq
	d_\varepsilon(y_j, b_j)\leq \int_{\alpha_j[y_j, b_j]} \rho_\varepsilon(u) \, ds \leq \rho_\varepsilon(x) e^{\varepsilon(H+h+1)} \int_0^\infty e^{-\varepsilon t} =\frac{1}{\varepsilon} \rho_\varepsilon(x) e^{\varepsilon(H+h+1)}.
	\eeqq
	Combining the above two inequalities, we infer that
	\beqq
	d_\varepsilon(x, b_j) \leq d_\varepsilon(x, y_j) + d_\varepsilon(y_j, b_j)\leq \frac{e^{\varepsilon(H+ \nu+1)}+e^{\varepsilon(H+ h+1)} - 1}{\varepsilon} \rho_\varepsilon(x).
	\eeqq
	
	As $\{b_j\}$ is a Gromov sequence, we know from Proposition~\ref{prop:boundary-id} that $b_j\stackrel{d_\varepsilon}{\longrightarrow}\zeta\in \partial X_\varepsilon$ as $j \to \infty$. Thus
	\[
	d_\varepsilon(x) \leq d_\varepsilon(x, \zeta)\leq \frac{e^{\varepsilon(H+ \nu+1)}+e^{\varepsilon(H+ h+1)} - 1}{\varepsilon} \rho_\varepsilon(x).
	\]
	Since $0 < \varepsilon \leq \varepsilon_1(\delta)$, Lemma~\ref{lem:varepsilon boundary distance} follows by setting $A_1=e^{\varepsilon_1(H+ \nu+1)}+e^{\varepsilon_1(H+ h+1)} - 1$.
\end{proof}



\begin{proof}[Proof of Proposition~\textup{\ref{prop:quasiisometric to quasisimilar}}]
	With Lemma \ref{lem:varepsilon boundary distance} at hand, the proof of Proposition~\textup{\ref{prop:quasiisometric to quasisimilar}} follows by a similar argument as that used in \cite[Proposition~4.15]{BHK}. Thus, we only outline the main steps below.
	
	Let $f \colon X \to X'$ be a quasiisometry between two complete roughly starlike Gromov hyperbolic spaces, with $X$ $\delta$-hyperbolic and $X'$ $\delta'$-hyperbolic. For  $0<\varepsilon\leq\varepsilon_1(\delta)$ and $0<\varepsilon'\leq\varepsilon_1(\delta')$ as in Theorem~\textup{\ref{thm:uniform}}, let $X_\varepsilon = (X, d_\varepsilon)$ and $X'_{\varepsilon'} = (X', d_{\varepsilon'})$ be the corresponding uniformized spaces. Choosing base points $w \in X$ and $w' \in X'$ with $f(w) = w'$, which is harmless up to quasiisometry. We need to prove that $f \colon X_\varepsilon \to X'_{\varepsilon'}$ is a quasisimilarity.
	
	Fix $z \in X$ and a sufficiently small $\lambda$ depending only on the data and on the ratio $\varepsilon'/\varepsilon$. For $x, y \in B_\varepsilon(z, \lambda d_\varepsilon(z))$, by \eqref{eq:twopoint} and \eqref{eq:varepsilon boundary distance}, we obtain
	\begin{equation}\label{eq:locally quasisimilarity 1}
		\frac{1}{C}\,d_\varepsilon(x,y) \;\le\; d_\varepsilon(z)\,\varepsilon|x-y| \;\le\; C\,d_\varepsilon(x,y),
	\end{equation}
	and
	\begin{equation}\label{eq:locally quasisimilarity 2}
		\frac{1}{C}\,d_{\varepsilon'}(x',y') \;\le\; d_{\varepsilon'}(z')\,\varepsilon'| x'-y'| \;\le\; C\,d_{\varepsilon'}(x',y'),
	\end{equation}
where $C\geq1$ depends only on the data associated with $X$ and $X'$, not necessarily the same at each occurrence. As $f \colon X \to X'$ is quasiisometric, the above two estimates imply that $f$ satisfies condition \eqref{eq:con2} for quasisimilarity. It remains to prove that $f$ is $\eta$-quasisymmetric.
	
	Since $X_\varepsilon$ and $X'_{\varepsilon'}$ are quasiconvex by Theorem~\ref{thm:uniform}, it suffices by \cite[Theorem~6.6]{Vaisala1999} to prove that for all distinct $x, y, z \in X$,
	\begin{equation}\label{eq:weakly quasisymmetric}
		\frac{d_\varepsilon(x,y)}{d_\varepsilon(x,z)} \le 1
\quad\text{implies}\quad		
		\frac{d_{\varepsilon'}(x',y')}{d_{\varepsilon'}(x',z')} \le C.
	\end{equation}
	Under the condition $d_\varepsilon(x,y) \le d_\varepsilon(x,z)$, \eqref{eq:twopoint} gives
	\begin{equation}\label{eq:quasisymmetric con}
		\exp\bigl\{\varepsilon(x|z)_w - \varepsilon(x|y)_w\bigr\}\,
		\frac{1 \wedge \varepsilon|x-y|}{1 \wedge \varepsilon|x-z|} \le C.
	\end{equation}
	
	For sufficiently small $\varepsilon|x-z|$, as in \eqref{eq:locally quasisimilarity 1}, we obtain 
	\[
	\frac{d_\varepsilon(x,y)}{d_\varepsilon(x)}\le
	\frac{d_\varepsilon(x,z)}{d_\varepsilon(x)}\le C\,\varepsilon|x-z|, \quad
	\frac{d_\varepsilon(x,z)}{d_\varepsilon(x)}\geq \frac 1C\,\varepsilon|x-z|.
	\] 
	Together with \eqref{eq:con2}, \eqref{eq:weakly quasisymmetric} follows. 
	
	In the remaining case where $C\varepsilon|x-z| \ge 1$, if $\varepsilon|x-y| < 1$, then \eqref{eq:weakly quasisymmetric} is an immediate consequence of \eqref{eq:twopoint}.
	If instead $\varepsilon|x-y| \ge 1$, then \eqref{eq:quasisymmetric con} implies $(x|z)_w \le (x|y)_w + C$. By the quasi-invariance of the Gromov product under quasiisometries (see \cite[Theorem~3.21]{Vaisala}), the image points satisfy
	\[
	C \le \varepsilon'|x'-z'|
	\quad\text{and}\quad
	(x'|z')_{w'} \le (x'|y')_{w'} + C.
	\]
	Consequently, by \eqref{eq:twopoint},
	\[
	\frac{d_{\varepsilon'}(x',y')}{d_{\varepsilon'}(x',z')}
	\le C\,\frac{1 \wedge \varepsilon'|x'-y'|}{1 \wedge \varepsilon'|x'-z'|}
	\le C,
	\]
	so \eqref{eq:weakly quasisymmetric} holds in this case as well.

\end{proof}

\subsection{Quasisimilarity of conformal deformation}
We now study the following composition a bounded $A$-uniform space $\Omega$:
$$\Omega \xrightarrow{\quad\mathcal{Q}\quad} (\Omega, k) \xrightarrow{\quad\mathcal{D}\quad} \Omega_\varepsilon.$$
Here we denote by $\Omega_\varepsilon$ the uniformization of $(\Omega, k)$ via the density $\rho_\varepsilon(x) = \exp\left(-\varepsilon k(w, x)\right)$, where $0 < \varepsilon \leq \varepsilon_1(A)$ and $w \in \Omega$ is a base point satisfying $d_\Omega(w) \ge \frac12\sup\limits_{z\in\Omega}d_\Omega(z) $.

\begin{prop}\label{prop:quasisimilar}
	The identity map $\Omega \to \Omega_\varepsilon$, $0 < \varepsilon \leq \varepsilon_1(A)$, is quasisimilar with constant depending only on $A$ and $\varepsilon$.
\end{prop}
We shall need the following auxiliary lemma for the proof of Proposition~\ref{prop:quasisimilar}.

\begin{lem}
	\label{lem:intrinsic-314}
	Let $(\Omega,d)$ be a bounded $A$-uniform space and $w\in\Omega$ satisfy \eqref{eq:deep-basepoint-road}.
	For $\mu>0$, let $\alpha\in \Lambda^{1,\mu}_{a,w}(\Omega,k)$
 be parameterized by $d$-arclength from $a$. For $b\in\Omega\setminus\{a\}$, define
	\[
	y
	=
	\begin{cases}
		\alpha\bigl(d(a,b)\bigr),
		& d(a,b)\le \dfrac12 d(a,w),\\[2mm]
		w,
		& d(a,b)>\dfrac12 d(a,w).
	\end{cases}
	\]
	Then there exist constants $C_1,C_2>0$, depending only on $A$ and $\mu$, such that, for every
	$\beta\in\Lambda^{1,\mu}_{a,b}(\Omega,k),$
	one has
	\begin{equation}\label{eq:314-a}
		k\bigl(y,\beta\bigr)
		\le
		C_1
	\end{equation}
	and
	\begin{equation}\label{eq:314-b}
		k(w,\beta)-C_2
		\le
		k\bigl(w,y\bigr)
		\le
		k(w,\beta)+C_2.
	\end{equation}
\end{lem}

\begin{proof} 
	
	Observe that if \eqref{eq:314-a} holds, the left-hand inequality of \eqref{eq:314-b} can be directly derived from the triangle inequality. So it suffices to prove \eqref{eq:314-a} and the right-hand inequality of \eqref{eq:314-b}.
	
	Assume first that
	$
	d(a,b)>\frac12d(a,w).
	$
	Then $y=w$. For $\beta\in\Lambda^{1,\mu}_{a,b}(\Omega,k),$ choose $x\in\beta$ so that
	\[
	\ell_d(\beta[a,x])=\frac14d(a,w).
	\]
	Since $\beta$ is a $B_0$-uniform curve by Theorem~\ref{thm:intrinsic-210} and $\ell_d(\beta)\ge d(a,b)>\frac12d(a,w),$
	we have
	\begin{equation*}
		d_\Omega(x)
		\ge \frac{1}{B_0}\ell_d\left(\beta[a,x]\right)\wedge \ell_d\left(\beta[x,b]\right)=
		\frac{1}{4B_0}d(a,w).
	\end{equation*}
	The base-point condition \eqref{eq:deep-basepoint-road} gives
	\[
	d_\Omega(x)
	\le
	\sup_{z\in\Omega}d_\Omega(z)
	\le
	2d_\Omega(w),
	\]
	and thus
	\begin{equation*}
		d_\Omega(w)
		\ge
		\frac{1}{8B_0}d(a,w).
	\end{equation*}
 Note also that 
	\[
	d(w,x)
	\le
	d(w,a)+d(a,x)
	\le
	\frac54d(a,w).
	\]
	By \eqref{eq:j-k-upper} and the above estimate, we obtain
	\[
	\begin{aligned}
		k(w,\beta)
		&\le k(w,x)\le
		4A^2\log\left(
		1+
		\frac{d(w,x)}
		{d_\Omega(w)\wedge d_\Omega(x)}
		\right)\le
		4A^2\log(1+10B_0).
	\end{aligned}
	\]
	Since $y=w$, both \eqref{eq:314-a} and \eqref{eq:314-b} follow in this case.
	
	Assume now that
	$
	d(a,b)\le\frac12d(a,w),
	$
	and $y=\alpha(d(a,b))$. Since $\alpha$ is parameterized by $d$-arc length, this point is well-defined. Moreover,
	\beq\label{eq:314-length}
	\ell_d(\alpha[a,y])=d(a,b),
	\quad
	\ell_d(\alpha[y,w])
	\ge d(a,w)-d(a,b)
	\ge d(a,b).
	\eeq
	Since $\alpha$ is $B_0$-uniform, we have
	\begin{equation}\label{eq:314-deltaq}
		d_\Omega(y)
		\ge \frac{1}{B_0}\ell_d\left(\alpha[a,y]\right)\wedge \ell_d\left(\alpha[y,w]\right)=
		\frac{1}{B_0}d(a,b).
	\end{equation}
    For $\beta\in\Lambda^{1,\mu}_{a,b}(\Omega,k),$	choose $x\in\beta$ bisecting the $d$-length. Since $\ell_d(\beta)\ge d(a,b)$ and $\beta$ is  $B_0$-uniform  by Theorem~\ref{thm:intrinsic-210}, it holds 
	\begin{equation*}
		d_\Omega(x)
		\ge \frac{1}{B_0}\ell_d\left(\beta[a,x]\right)\wedge \ell_d\left(\beta[x,b]\right)=
		\frac{1}{2B_0}d(a,b).
	\end{equation*}
	Since $\beta$ is $B_0$-quasiconvex in $(\Omega,d)$, we have 
	\[
	d(x,y)
	\le
	d(x,a)+d(a,y)
	\le
	\frac12\ell_d(\beta)+d(a,b)
	\le
	\left(\frac{B_0}{2}+1\right)d(a,b).
	\]
	These imply
	\[
	k(y,\beta)
	\le
	k(y,x)
	\stackrel{\eqref{eq:j-k-upper}}{\le}
	4A^2\log\bigl(1+B_0(B_0+2)\bigr).
	\]
	This proves \eqref{eq:314-a}, 
	and thus also the left-hand inequality in \eqref{eq:314-b}.
	
	For the right-hand inequality in \eqref{eq:314-b}, choose $u\in\beta$ so that $k(w,u)=k(w,\beta).$
	
	If
$
	\ell_d(\beta[a,u])
	\wedge
	\ell_d(\beta[u,b])
	\ge
	\frac12d(a,b),
$
	then $$d_\Omega(u)\ge\frac{1}{2B_0}d(a,b),$$
	and
	\[  
	d(u,y)
	\le d(u,a)+d(a,y)\leq\ell_d(\beta)+d(a,b)\leq
	(B_0+1)d(a,b).
	\]
	Then it follows from \eqref{eq:j-k-upper} and \eqref{eq:314-deltaq} that
	\[
	k(w,y)
	\le
	k(w,u)+k(u,y)
	\le
	k(w,\beta)+4A^2\log\bigl(1+2B_0(B_0+1)\bigr).
	\]
	
	Next, we consider the case
	\[
	\ell_d(\beta[a,u])
	\wedge
	\ell_d(\beta[u,b])
	<
	\frac12d(a,b).
	\]
	Choose $\beta_0\in\Lambda^{1,\mu}_{u,w}(\Omega,k).$
	If $\ell_d\left(\beta[a,u]\right)<\frac{1}{2}d(a,b)$, then
	\[
	d(u,w)
	\ge
	d(a,w)-d(a,u)
	>
	2d(a,b)-\frac12d(a,b)=\frac{3}{2}d(a,b).
	\]
	If $\ell_d\left(\beta[u,b]\right)<\frac{1}{2}d(a,b)$, then
	\[
	d(u,w)
	\ge
	d(a,w)-d(a,b)-d(b,u)
	>
	2d(a,b)-d(a,b)-\frac12d(a,b)=\frac12d(a,b).
	\]
	In either case, we have shown that $d(u,w)>\frac12d(a,b)$, and thus $\ell_d(\beta_0)>\frac12d(a,b)$.
	Choose $u'\in\beta_0$ so that
	\beq\label{eq:314-length-u'}
	\ell_d(\beta_0[u,u'])=\frac14d(a,b).
	\eeq
	Since $\beta_0$ is $B_0$-uniform by Theorem~\ref{thm:intrinsic-210}, it holds
	\begin{equation}
		\label{eq:314-deltauprime}
		d_\Omega(u')\ge \frac{1}{B_0}\ell_d\left(\beta_0[u,u']\right)\wedge \ell_d\left(\beta_0[u',w]\right)
		=\frac{1}{4B_0}d(a,b).
	\end{equation}
	On the other hand, since $\beta$ is also $B_0$-uniform by Theorem~\ref{thm:intrinsic-210}, \eqref{eq:314-length} and \eqref{eq:314-length-u'} give
	\[
	d(u',y)\leq d(u',u)+d(u,a)+d(a,y)\leq \left(B_0+\frac54\right)d(a,b).
	\]
	Together with \eqref{eq:314-deltaq} and
	\eqref{eq:314-deltauprime}, this yields
	\[
	k(u',y)
	\le
	4A^2\log\bigl(1+B_0(4B_0+5)\bigr).
	\]
	Finally, since $\beta_0$ is a $(1,\mu)$-quasigeodesic in $(\Omega,k)$, the above estimate implies 
	\[
	\begin{aligned}
		k(w,y)
		&\le 
		k(w,u')+k(u',y)\le
		\ell_k(\beta_0)+k(u',y)\le
		k(w,u)+\mu+k(u',y)\\
		&\le
		k(w,u)+\mu+4A^2\log\bigl(1+B_0(4B_0+5)\bigr)\\
		&=
		k(w,\beta)+\mu+4A^2\log\bigl(1+B_0(4B_0+5)\bigr).
	\end{aligned}
	\]
	In both cases, we have proved the right-hand inequality in \eqref{eq:314-b}.
\end{proof}


\begin{proof}[Proof of Proposition~\textup{\ref{prop:quasisimilar}}]
	First, note that by Theorem \ref{thm:Gromov and roughly starlike}, $(\Omega,k)$ is a complete $\delta=\delta(A)$-Gromov hyperbolic space. Let $\varepsilon_1(A)=\varepsilon_0(\delta(A),1,1)$ be the constant given by Theorem~\ref{thm:uniform}.	We shall verify separately the local bi-Lipschitz condition \eqref{eq:con2} and the quasisymmetry of the identity map
	$\operatorname{id}:\Omega\longrightarrow\Omega_\varepsilon$ for $0<\varepsilon\leq\varepsilon_1(A)$.
	
	Choose $\lambda=\lambda(A)\in(0,1/4)$ small so that
	\begin{equation}\label{6.9}
		4A^2\log\left(1+\frac{2\lambda}{1-\lambda}\right)\leq \frac{1}{\varepsilon_1(A)}.
	\end{equation}
	If $y,z\in B_d\bigl(x,\lambda d_\Omega(x)\bigr),$ then
	\[
	d_\Omega(y)\wedge d_\Omega(z)\geq(1-\lambda)d_\Omega(x),
	\quad
	d(y,z)\leq2\lambda d_\Omega(x),
	\]
	and hence
	\begin{equation}\label{6.10}
		k(y,z)\stackrel{\eqref{eq:j-k-upper}}{\le}
		4A^2\log\left(1+\frac{d(y,z)}{d_\Omega(y)\wedge d_\Omega(z)}\right) \leq\frac{1}{\varepsilon_1(A)}.
	\end{equation}
	The same argument gives $k(x,y)+k(x,z)\leq\frac{2}{\varepsilon_1(A)}$. Consequently,
	\[
	\bigl|(y|z)_w-k(w,x)\bigr|\leq C,
	\]
	where $C=C(A)$ is a constant,
	and the two-point estimate in Lemma~\ref{lem:two-point} yields
	\begin{equation}\label{6.11}
		d_\varepsilon(y,z)\asymp_{A}
		\rho_\varepsilon(x)\,k(y,z).
	\end{equation}
	On the other hand, \eqref{eq:quasihyperbolic inequality}, \eqref{eq:j-k-upper}, and the elementary comparability of $\log(1+t)$ and $t$ on bounded intervals give
	\begin{equation}\label{6.12}
		k(y,z)\asymp_{A}\frac{d(y,z)}{d_\Omega(x)}.
	\end{equation}
	Combining \eqref{6.11} and \eqref{6.12}, we obtain
	\begin{equation}\label{6.13}
		d_\varepsilon(y,z)
		\asymp_{A}
		\frac{\rho_\varepsilon(x)}{d_\Omega(x)}\,d(y,z).
	\end{equation}
	Thus the condition \eqref{eq:con2} holds with local scale
	$c_x=\frac{\rho_\varepsilon(x)}{d_\Omega(x)}.$
	
	It remains to prove that the identity map is quasisymmetric. By \cite[Theorem 6.6]{Vaisala1999}, it is enough to prove that for distinct $a,b,c\in\Omega$,
	\begin{equation}\label{6.14}
		d(a,b)\leq d(a,c)
		\quad\text{implies}\quad
		d_\varepsilon(a,b)\leq C(A,\varepsilon)\,d_\varepsilon(a,c).
	\end{equation}
Choose $\alpha\in\Lambda^{1,1}_{a,w}(\Omega,k),\,
	\beta_b\in\Lambda^{1,1}_{a,b}(\Omega,k),\,
	\beta_c\in\Lambda^{1,1}_{a,c}(\Omega,k),$
	and let $y_b,y_c\in\alpha$ be the points associated with $b,c$,
	respectively, as in Lemma~\ref{lem:intrinsic-314}. By Lemma~\ref{lem:intrinsic-314} and Theorem~\ref{thm:intrinsic-36}, we have 
	\begin{equation}\label{6.15}
		\bigl|(a|b)_w-k(w,y_b)\bigr|\leq C
		\quad \text{and}\quad 
		\bigl|(a|c)_w-k(w,y_c)\bigr|\leq C.
	\end{equation}
	Since $d(a,b)\leq d(a,c)$, the point $y_b$ is encountered before
	$y_c$ when traveling from $a$ to $w$ along $\alpha$. Since $\alpha$
	is a quasihyperbolic $(1,1)$-quasigeodesic in $(\Omega,k)$, it holds 
	\begin{equation}\label{6.16}
		k(w,y_b)-k(w,y_c)\geq k(y_b,y_c)-1.
	\end{equation}	
	The two-point estimate of Lemma~\ref{lem:two-point},  \eqref{6.15} and \eqref{6.16} therefore give
	\begin{equation}\label{6.17}
		\frac{d_\varepsilon(a,b)}{d_\varepsilon(a,c)}
		\leq
		C e^{-\varepsilon k(y_b,y_c)}
		\frac{1\wedge\varepsilon k(a,b)}
		{1\wedge\varepsilon k(a,c)}.
	\end{equation}

	Note that since 
	\[
	d(a,b)\leq d(a,c)\leq d_{\Omega}(a)\left(e^ {k(a,c)}-1\right),
	\]
	we may choose $c_0=c_0(A)>0$ sufficiently small so that if $k(a,c)\leq c_0$, then $b,c\in B_d\bigl(a,\lambda d_\Omega(a)\bigr).$
	In this case, \eqref{6.14} follows immediately from the local comparison
	\eqref{6.13}.
	It remains to consider the case where $k(a,c)>c_0$. Since
	$
	1\wedge\varepsilon k(a,c)
	\geq
	1\wedge\varepsilon c_0,
	$
	while $e^{-\varepsilon k(y_b,y_c)}\leq1$, \eqref{6.17} gives 
	\[
	\frac{d_\varepsilon(a,b)}{d_\varepsilon(a,c)}
	\leq
	\frac{C}{1\wedge\varepsilon c_0}.
	\]
	Hence the identity map is quantitatively quasisymmetric. 
	The proof of Proposition~\ref{prop:quasisimilar} is complete.
\end{proof}


Next, we construct a counter-example, which shows that the quasisimilarity constant cannot be uniform with respect to $\varepsilon$. This shows in particular that \cite[Proposition 4.28]{BHK} is inaccurately stated. More precisely, we have 
\begin{exam}\label{exam:counter-example}
	Let
	\[
	\Omega=\mathbb B^2
	=
	\{x\in\mathbb R^2:|x|<1\}
	\]
	be the Euclidean unit disk, equipped with the Euclidean metric $d(x,y)=|x-y|.$
	Then $\Omega$ is a bounded uniform domain. Let $w=0$ and for $0<\varepsilon<1$, let $k$ denote the quasihyperbolic metric on $\Omega$. Define
	\[
	\rho_\varepsilon(x)
	=
	e^{-\varepsilon k(w,x)}.
	\]
	Let $d_\varepsilon$ be the conformal deformation of the
	quasihyperbolic metric:
	\[
	d_\varepsilon(x,y)
	=
	\inf_\gamma
	\int_\gamma
	\rho_\varepsilon\,ds_k,
	\quad
	ds_k=\frac{ds}{\delta},
	\]
	where the infimum is taken over all rectifiable curves in $(\Omega,d)$ joining $x$, $y$, and
	\[
	\delta(x)=\operatorname{dist}(x,\partial\Omega)=1-|x|.
	\]
	
	Then there is no quasisymmetry function
	$
	\eta:[0,\infty)\to[0,\infty),
	$
	independent of $\varepsilon$, such that
	\[
	\frac{d_\varepsilon(x,y)}
	{d_\varepsilon(x,z)}
	\leq
	\eta\left(
	\frac{d(x,y)}{d(x,z)}
	\right)
	\]
	for all distinct $x,y,z\in\Omega$ and all sufficiently small
	$\varepsilon>0$.
\end{exam}

\begin{proof}
	
	For $x\in\mathbb B^2$, we have
$
	\delta(x)=1-|x|.
	$
	Let $x=r\theta,\,
	0\leq r<1,\,
	|\theta|=1.$
	Any radial segment is a quasihyperbolic geodesic with
	\[
	k(0,x)=-\log(1-|x|).
	\]	
	Consequently,
	\[
	\rho_\varepsilon(x)
	=
	e^{-\varepsilon k(0,x)}
	=
	(1-|x|)^\varepsilon.
	\]
	Since
	$
	ds_k=\frac{ds}{1-|x|},
	$
	the deformed length element is
	\[
	\rho_\varepsilon\,ds_k
	=
	(1-|x|)^{\varepsilon-1}\,ds.
	\]
	Direct computation shows that radial segments are $d_\varepsilon$-geodesics: if $0\leq r<s<1$ and $\theta\in S^1$, then
	\begin{equation}\label{eq:CE 4}
		d_\varepsilon(r\theta,s\theta)
		=
		\int_r^s(1-t)^{\varepsilon-1}\,dt.
	\end{equation}

	For $r,s\in(0,1)$, consider the two points $r\theta$ and $-s\theta$. Then the diameter segment joining them through the origin is a $d_\varepsilon$-geodesic with 	
	\begin{equation}\label{eq:CE 6}
		d_\varepsilon(r\theta,-s\theta)
		=	\int_0^r(1-t)^{\varepsilon-1}\,dt
		+
		\int_0^s(1-t)^{\varepsilon-1}\,dt.
	\end{equation}

	Fix
	$
	x=\left(\frac12,0\right),
	z=\left(-\frac45,0\right),
	$
	and for $0<t<1/2$, set
	$
	y_t=(1-t,0).
	$
	Then
	\[
	\frac{d(x,y_t)}{d(x,z)}
	=
	\frac{10}{13}\left(\frac12-t\right)
	<1.
	\]
	Moreover,
	\[
	\lim_{t\downarrow0}
	\frac{d(x,y_t)}{d(x,z)}
	=
	\frac5{13}.
	\]

	Both $x$ and $y_t$ lie on the positive $x_1$-axis. Hence, by \eqref{eq:CE 4}, it holds 
	\[
	d_\varepsilon(x,y_t)
	=
	\int_{1/2}^{1-t}
	(1-r)^{\varepsilon-1}\,dr		=
	\frac{(1/2)^\varepsilon-t^\varepsilon}{\varepsilon}.
	\]

	The points $x$ and $z$ lie on opposite sides of the origin.
	By \eqref{eq:CE 6} and \eqref{eq:CE 4}, we have
	\[
	d_\varepsilon(x,z)
	=
	d_\varepsilon(x,0)
	+
	d_\varepsilon(0,z)=\frac{
		2-2^{-\varepsilon}-5^{-\varepsilon}
	}{\varepsilon}.
	\]

	Combining the above two estimates gives 
	\[
	\frac{d_\varepsilon(x,y_t)}
	{d_\varepsilon(x,z)}
	=
	\frac{
		2^{-\varepsilon}-t^\varepsilon
	}{
		2-2^{-\varepsilon}-5^{-\varepsilon}
	}.
	\]
	In particular, it holds
	\begin{equation}\label{eq:CE 16}
		\lim_{t\downarrow0}
		\frac{d_\varepsilon(x,y_t)}
		{d_\varepsilon(x,z)}
		=
		\frac{
			2^{-\varepsilon}
		}{2-2^{-\varepsilon}-5^{-\varepsilon}}.
	\end{equation}

	Note that if $\varepsilon$ is small, then
	$$	\frac{
		2^{-\varepsilon}
	}{
		2-2^{-\varepsilon}-5^{-\varepsilon}
	}
	=
	\frac{1}{\varepsilon\log10}+O(1)
	\quad
	(\varepsilon\downarrow0).
	$$
	In particular, we have
	\begin{equation}\label{eq:CE 20}
		\lim_{\varepsilon\downarrow0}
		\lim_{t\downarrow0}
		\frac{d_\varepsilon(x,y_t)}
		{d_\varepsilon(x,z)}
		=
		+\infty.
	\end{equation}

	Suppose, for contradiction, that there exists a quasisymmetry function
	$
	\eta:[0,\infty)\to[0,\infty)
	$
	which is independent of $\varepsilon$ and such that
	\[
	\frac{d_\varepsilon(u,v)}
	{d_\varepsilon(u,q)}
	\leq
	\eta\left(
	\frac{d(u,v)}
	{d(u,q)}
	\right)
	\]
	for all distinct $u,v,q\in\Omega$ and all $0<\varepsilon\leq\varepsilon_0$.
	
	Apply the $\eta$-quasisymmetric condition with $u=x$, $v=y_t$ and $q=z$. By the preceding computation,
	\[
	\frac{d(x,y_t)}{d(x,z)}<1.
	\]
	Therefore, for every $t\in(0,1/2)$ and every
	$0<\varepsilon\leq\varepsilon_0$, we have
	\[
	\frac{d_\varepsilon(x,y_t)}
	{d_\varepsilon(x,z)}
	\leq
	\eta(1).
	\]
	Letting $t\downarrow0$, we obtain from \eqref{eq:CE 16} and the above estimate that
	\[
	\frac{
		2^{-\varepsilon}
	}{
		2-2^{-\varepsilon}-5^{-\varepsilon}
	}
	\leq
	\eta(1)<\infty. 
	\]
	But by \eqref{eq:CE 20}, the left-hand side tends to $+\infty$ as
	$\varepsilon\downarrow0$. This contradicts the finiteness of
	$\eta(1)$.
	
	Hence, no quasisymmetry function can be chosen independently of 	$\varepsilon$.
\end{proof}

\subsection{Property of $\mathcal{Q}$}
To finish the proof of Theorem~\ref{thm:one-to-one correspondence}, two assertions remain to be verified. First, the map~$\mathcal{Q}$ sends mutually quasisimilar uniform spaces into a single quasiisometry class. Second, composing the dampening deformation~$\mathcal{D}$ with the quasihyperbolization~$\mathcal{Q}$,
\[
(X, |x-y|) \xrightarrow{\quad\mathcal{D}\quad} (X_\varepsilon, d_\varepsilon) \xrightarrow{\quad\mathcal{Q}\quad} (X_\varepsilon, k_\varepsilon),
\]
returns a space that remains in the same quasiisometry class as the original. Here $X$ denotes a complete roughly starlike $\delta$-Gromov hyperbolic space, and the parameter~$\varepsilon$ is chosen in the range $0 < \varepsilon \leq \varepsilon_1(\delta)$ according to Theorem~\ref{thm:uniform}. Both of these assertions are proved in Propositions~\ref{prop:quasisimilar to quasiisometry} and~\ref{prop:quasiisometry}, respectively.

\begin{prop}\label{prop:quasisimilar to quasiisometry}
	A quasisimilarity between two uniform spaces is a quasiisometry with respect to the quasihyperbolic metrics, quantitatively.
\end{prop}

\begin{proof}
	Let	$(\Omega, d)$ be $A$-uniform, $(\Omega', d')$ be $A'$-uniform, and let $f\colon (\Omega, d) \longrightarrow (\Omega', d')$ be a quasisimilar map with data $(\eta, L, \lambda)$.	To prove $f\colon (\Omega, k) \longrightarrow (\Omega', k')$ is a quasiisometry, it suffices to verify that $f$ is Lipschitz by symmetry. Fix $x, y \in \Omega$. We consider the following two cases.
	\smallskip 
	
\textbf{Case 1:} $d(x,y) < \lambda \, (d_{\Omega}(x) \wedge d_{\Omega}(y))$.
\smallskip 

In this case, we claim that
\beq\label{5.11}
d'(f(x), f(y)) \leq C_1 \, \frac{d_{\Omega'}'(f(z))}{d_{\Omega}(z)} \, d(x,y) \quad \text{for } z = x \text{ or } y,	
\eeq
where $C_1=C_1(\eta,L, \lambda)\geq1.$

We only consider the case for $z = x$, as the proof of other case is similar. 
By the quasisimilarity condition \eqref{eq:quasisymmetric} and \eqref{eq:quasisymmetric}, for any $u, v \in \partial B(x, \frac{\lambda}{2} d_\Omega(x))$, it holds 
\[
\frac{d'(f(u), f(x))}{d'(f(v), f(x))} \leq \eta\left(\frac{d(u,x)}{d(v,x)}\right) = \eta(1).
\]
Hence there exists $x_0 \in \partial B(x, \frac{\lambda}{2} d_{\Omega}(x))$ such that  $d'(f(x_0), f(x)) \leq \eta(1) \, d_{\Omega'}'(f(x))$ and so
$$\frac{d'(f(x), f(y))}{d'(f(x), f(x_0))} \geq \frac{d'(f(x), f(y))}{\eta(1) \, d_{\Omega'}'(f(x))}.$$
On the other hand, by \eqref{eq:con2}, we have 
\beqq
\frac{d'(f(x), f(y))}{d'(f(x), f(x_0))} \leq \frac{L^2 \, d(x,y)}{d(x, x_0)} = \frac{2L^2 \, d(x,y)}{\lambda \, d_{\Omega}(x)}.
\eeqq
Combining these estimates gives 
\beq
d'(f(x), f(y)) \leq C_1\, \frac{d_{\Omega'}'(f(x))}{d_{\Omega}(x)} \, d(x,y),
\eeq
where $C_1=C_1(\eta,L, \lambda)\geq1.$ This proves \eqref{5.11}.

By Lemma~\ref{lem:intrinsic-213} and \eqref{5.11}, we have 
\beq\label{5.12}
\begin{aligned}
	k'(f(x), f(y)) &\leq 4A'^2 \log\left(1 + \frac{d'(f(x), f(y))}{d_{\Omega'}'(f(x)) \wedge d_{\Omega'}'(f(y))}\right) \\
	&\leq 4A'^2 \log\left(1 + \frac{C_1 \, d(x,y)}{d_{\Omega}(x) \wedge d_{\Omega}(y)}\right) \\
	&\leq 4A'^2 C_1 \, \log\left(1 + \frac{d(x,y)}{d_{\Omega}(x) \wedge d_{\Omega}(y)}\right) \\
	&\leq 4A'^2 C_1 \, k(x,y),
\end{aligned}
\eeq
where in the second last inequality, we used the following elementary inequality 
\beqq 
\log(1+tx)\leq t\log(1+x), \quad t\geq1,\,x\geq 0.
\eeqq	

\textbf{Case 2:} $d(x,y) \geq \lambda \, (d_{\Omega}(x) \wedge d_{\Omega}(y))$.
\smallskip 
	
	For any $\varsigma > 0$, select a curve $\gamma$ joining $x$ and $y$ so that  $\ell_k(\gamma) \leq k(x,y) + \varsigma$, and a point $y_0 \in \gamma$ so that $d_\Omega(y_0) < 2 \inf\limits_{z \in \gamma} d_\Omega(z)$.
	
	Set $x_0 = x$ and for $i=1,\cdots,n$, let $x_i$ be the last point of $\gamma[x_{i-1}, y] \cap \overline{B}(x_{i-1}, \frac{1}{2} \lambda \, d(y_0))$ along the curve $\gamma$ from $x$ to $y$, where $x_{n}=y$, then for $i=1,\cdots,n$,
	$$d(x_{i-1}, x_i) \leq \frac{1}{2} \lambda \, d_{\Omega}(y_0) < \lambda \, (d_{\Omega}(x_{i-1}) \wedge d_{\Omega}(x_i)).$$
	Using the estimate \eqref{5.12} from Case 1, we obtain 
	\beqq
	\begin{aligned}
		k'(f(x), f(y)) &\leqslant \sum_{i=1}^{n} k'(f(x_{i-1}), f(x_i))\leq \sum_{i=1}^{n} 4A'^2 C_1 \, k(x_{i-1}, x_i)\\
		&\leq 4A'^2 C_1 \, \ell_k(\gamma)\leq 4A'^2 C_1 \, (k(x,y) + \varsigma).
	\end{aligned}
	\eeqq
	Sending $\varsigma$ to $0$, we have
	$$k'(f(x), f(y)) \leqslant 4A'^2 C_1 \, k(x,y).$$
	
	This completes the proof.
\end{proof}
\begin{prop}\label{prop:quasiisometry}
	If $X$ is complete $\delta$-Gromov hyperbolic and $(H, \nu, h)$-roughly starlike with respect to $w$, then for $0 < \varepsilon \leq \varepsilon_1(\delta)$ as in Theorem~\textup{\ref{thm:uniform}} and any $ x, y \in X$, we have
	\[
	c' \varepsilon |x - y| \leq k_\varepsilon(x, y) \leq  2e \varepsilon |x - y|,
	\]
	where $c' \in (0, 1)$ depends only on the given data. In particular,
	the identity map $(X, \varepsilon |x-y|) \to (X_\varepsilon, k_\varepsilon)$ is quasiisometric.
\end{prop}

\begin{proof}
	Let  $\gamma \in \Lambda_{x,y}^{1,1}(X,d)$ be parametrized by $d$-arclength from $x$ and let 
	$$L_\varepsilon(t) = \ell_\varepsilon(\gamma[0,t]) = \int_0^t \rho_\varepsilon(\gamma(s)) \, ds.$$ 
	Then by \eqref{eq:varepsilon boundary distance} and \eqref{eq:prop of curve}, we have 
	\[
	k_\varepsilon(x, y) \leq \int_\gamma \frac{|d_\varepsilon z|}{d_\varepsilon(z)} = \int_0^{\ell_d(\gamma)} \frac{d L_\varepsilon(t)}{d_\varepsilon(\gamma(t))} = \int_0^{\ell_d(\gamma)} \frac{\rho_\varepsilon(\gamma(t))}{d_\varepsilon(\gamma(t))} \, dt \leq 2e \varepsilon  |x - y|.
	\]
	
	In the other direction, it follows from the Gehring-Hayman inequality in Theorem~\ref{thm:G-H} and \eqref{eq:varepsilon boundary distance} that 
	\beq\label{5.13}
	\begin{aligned}
		k_\varepsilon(x, y) &\geq \log\left(1 + \frac{d_\varepsilon(x, y)}{d_\varepsilon(x) \wedge d_\varepsilon(y)}\right)
		\geq \log\left(1 + \frac{\ell_\varepsilon(\gamma)}{K(d_\varepsilon(x) \wedge d_\varepsilon(y))}\right)\\
		&\geq \log\left(1 + \frac{\varepsilon \ell_\varepsilon(\gamma)}{K A_1(\rho_\varepsilon(x) \wedge \rho_\varepsilon(y))}\right),
	\end{aligned}
	\eeq
	where $K = K(\delta)\geq1$ and $A_1=A_1(\delta, H, \nu, h)\geq1$.
	
	In the following, we estimate $\ell_\varepsilon(\gamma)$ and $\rho_\varepsilon(x) \wedge \rho_\varepsilon(y)$ in two cases.
	\smallskip 
	
	\textbf{Case 1:} $\varepsilon |x - y| \leq 2 \log 2$.
	\smallskip 
	
	By \eqref{eq:prop of curve}, it holds 
	\beqq
	\begin{aligned}
		\ell_\varepsilon(\gamma) &= \int_\gamma \rho_\varepsilon \, ds \geq\int_0^{\ell_d(\gamma)}\rho_\varepsilon(x)e^{-\varepsilon t}\,dt\\
		&\geq\rho_\varepsilon(x)\exp\left(-2\varepsilon  |x - y|\right)\,\ell_d(\gamma)\\
		&\geq \frac{1}{16} \rho_\varepsilon(x) |x - y|.
	\end{aligned}
	\eeqq
	From this,  \eqref{5.13}, and using $a \log(1 + t) \geq t \log(1 + a)$ for $0 \leq t \leq a$, we obtain
	\beq
	\begin{aligned}
		k_\varepsilon(x, y) &\geq \log\left(1 + \frac{\varepsilon \rho_\varepsilon(x) |x - y|}{ 16K A_1(\rho_\varepsilon(x) \wedge \rho_\varepsilon(y))}\right) \\
		&\geq \log\left(1 + \frac{\varepsilon |x - y|}{16 K A_1}\right) \geq \frac{\varepsilon\log 2}{16K A_1} |x - y|\\
		& \geq c_1 \varepsilon |x - y|,
	\end{aligned}
	\eeq
	where $c_1=c_1(\delta, H, \nu, h) \in (0, 1)$.
	\smallskip
	
	\textbf{Case 2:} $\varepsilon |x - y| > 2 \log 2$.
	\smallskip 
	
	Let  $\alpha:x\curvearrowright y$ be $2$-short. Let $x_\alpha\in \alpha$ be the point induced by $w$ as in Definition~\ref{def:induced subdivision}, by Lemma~\ref{thm:stability},   
	there exists $x_1 \in \gamma$ so that 
	\beq\label{5.14}
	|x_\alpha - x_1| \leq M_\delta=M(\delta,1,1,2). 
	\eeq
	Applying \cite[Lemma 4.1]{Allu} (with $p=w$, $z=x_\alpha$ and $u=x$ or $y$), we obtain
	\beq
	\begin{aligned}
		|w - x| &\geq |w - x_\alpha| + |x - x_\alpha| - 8\delta - 16, \\
		|w - y| &\geq |w - x_\alpha| + |y -x_\alpha| - 8\delta - 16.
	\end{aligned}
	\eeq
	It follows from \eqref{5.14} that 
	\beqq
	\begin{aligned}
		|w - x| \vee |w - y| &\geq |w - x_\alpha| + |x - x_\alpha| \vee |y - x_\alpha| - 8\delta - 16 \\
		&\geq |w - x_1| + \frac{1}{2}|x - y|- M_\delta - 8\delta - 16 .
	\end{aligned}
	\eeqq
	Then
	\beqq
	\begin{aligned}
		\rho_\varepsilon(x) \wedge \rho_\varepsilon(y) &= \exp\left(-\varepsilon (|w - x| \vee |w - y|)\right) \\
		&\leq e^{\varepsilon( M_\delta+8\delta + 16)} e^{-\varepsilon |w - x_1| - \frac{\varepsilon}{2}|x - y|} \\
		&= e^{\varepsilon(M_\delta+8\delta + 16)} \rho_\varepsilon(x_1) e^{-\frac{\varepsilon}{2}|x - y|}.
	\end{aligned}
	\eeqq
	
	Next, we estimate  $\ell_\varepsilon(\gamma).$
	Since $x_1 \in \gamma$ and $\max\{|x - x_1|,|y - x_1|\}\geq \frac{1}{2}|x-y|,$ we obtain
	\beqq
	\begin{aligned}
		\ell_\varepsilon(\gamma) &= \ell_\varepsilon(\gamma[x, x_1]) + \ell_\varepsilon(\gamma[x_1, y]) \\
		&\geq \rho_\varepsilon(x_1) \left(\int_0^{|x - x_1|} e^{-\varepsilon t} \, dt + \int_0^{|y - x_1|} e^{-\varepsilon t} \, dt\right) \\
		&= \frac{\rho_\varepsilon(x_1)}{\varepsilon} \left(2 - e^{-\varepsilon |x - x_1|} - e^{-\varepsilon |y - x_1|}\right) \\
		&\geq \frac{\rho_\varepsilon(x_1)}{\varepsilon} \left(1 - e^{-\frac{\varepsilon}{2}|x - y|}\right) > \frac{\rho_\varepsilon(x_1)}{2\varepsilon}.
	\end{aligned}
	\eeqq
	
	Combining these two estimates with \eqref{5.13} gives  
	\beqq
	k_\varepsilon(x, y) \geq \log\left(1 + c_2 \, e^{\frac{\varepsilon}{2}|x - y|}\right) 
	\geq c_2 \log\left(1 + e^{\frac{\varepsilon}{2}|x - y|}\right) 
	\geq   \frac{c_2\,\varepsilon}{2} |x - y|,
	\eeqq
	where $c_2=c_2(\delta, H, \nu, h) \in (0, 1)$ and in the second inequality, we used the elementary inequality
	\beqq 
	\log(1+C_1x)\geq C_1\log(1+x), \quad C_1\leq1,\,x\geq 0.
	\eeqq	
	
	The proof of Proposition~\ref{prop:quasiisometry} is complete. 
\end{proof}

\subsection{Proof of Theorem~\textup{\ref{thm:one-to-one correspondence}}}
\begin{proof}[Proof of Theorem~\textup{\ref{thm:one-to-one correspondence}}]
To prove Theorem~\ref{thm:one-to-one correspondence}, it suffices to show that the 
constructions $\mathcal{D}$ and $\mathcal{Q}$ induce mutually inverse 
maps between the corresponding equivalence classes. 
Proposition~\ref{prop:quasiisometric to quasisimilar} gives the forward 
direction: $\mathcal{D}$ sends quasiisometric complete roughly 
starlike Gromov hyperbolic spaces to quasisimilar bounded uniform 
spaces. Conversely, Proposition~\ref{prop:quasisimilar to 
	quasiisometry} gives the reverse direction: $\mathcal{Q}$ lifts 
quasisimilarities of bounded uniform spaces to quasiisometries of complete roughly starlike Gromov hyperbolic spaces. 
It remains to verify that the two constructions are compatible. This 
is precisely the content of Propositions~\ref{prop:quasisimilar} 
and~\ref{prop:quasiisometry}, which show that the compositions 
$\mathcal{D} \circ \mathcal{Q}$ and $\mathcal{Q} \circ \mathcal{D}$ 
are equivalent to the respective identities. Theorem~\ref{thm:one-to-one correspondence} now follows.
\end{proof}


\appendix

\section{Proof of Proposition \ref{prop:boundary-id}}

As in \cite[Theorem~1.3]{AlluJose} and \cite[Proposition~4.13]{BHK}, to prove Proposition~\ref{prop:boundary-id}, it suffices to prove the following lemma, which establishes a two-point estimate via quasigeodesics.
\begin{lem}\label{lem:two-point}
	Let $X$ be a $\delta$-Gromov hyperbolic space and $X_\varepsilon$ be its uniformization
	for $0<\varepsilon\leq\varepsilon_0$, where $\varepsilon_0=\varepsilon_0(\delta,1,1,1)$ is given by Theorem~\ref{thm:G-H}. Then there exists a constant $C_0=C_0(\delta)\geq1$ such that for all $x,y\in X$,
	\begin{equation}\label{eq:twopoint}
		\frac1{C_0}\,\frac{e^{-\varepsilon(x|y)_p}}{\varepsilon}
		1\wedge\varepsilon|x-y|
		\leq d_\varepsilon(x,y)
		\leq C_0\,\frac{e^{-\varepsilon(x|y)_p}}{\varepsilon}
		1\wedge\varepsilon|x-y|.
	\end{equation}
\end{lem}

\begin{proof}
	Let $\gamma \in \Lambda_{x,y}^{1,1}(X,d)$ and let $\alpha:x\curvearrowright y$, $\alpha_1:x\curvearrowright p$ be $2$-short arcs. Let $x_\alpha\in\alpha$ and $x_{\alpha_1},p_{\alpha_1}\in\alpha_1$ be the induced points, so that
	\[
	\ell_d(\alpha[x,x_\alpha])=(y|p)_x, \quad\ell_d(\alpha_1[x,x_{\alpha_1}])=(y|p)_x, \quad\ell_d(\alpha_1[p_{\alpha_1},p])=(x|y)_p.
	\]
	By \cite[Tripod Lemma 2.15]{Vaisala},
	\beq\label{5.0}
	|x_\alpha-x_{\alpha_1}|\leq 4\delta+4.
	\eeq
	From Lemma~\ref{thm:stability}, there exists $x_0\in\gamma$ such that
	\beq\label{5.1}
	|x_0-x_\alpha|\leq M_\delta=M(\delta,1,1,2).
	\eeq
	Since  $$|x_{\alpha_1}-p_{\alpha_1}|\leq \ell_d(\alpha_1[x_{\alpha_1},p_{\alpha_1}])=\ell_d(\alpha_1)-(y|p)_x-(x|y)_p=\ell_d(\alpha_1)-|x-p|\leq2,$$ together with \eqref{5.0} and \eqref{5.1}, it implies that
	\beq\label{5.2}
	\begin{aligned}
		|x_0-p|&\leq |x_\alpha-p|+|x_0-x_\alpha|\leq|x_{\alpha_1}-p|+|x_\alpha-x_{\alpha_1}|+M_\delta\\
		&\leq |x_{\alpha_1}-p_{\alpha_1}|+|p_{\alpha_1}-p|+ M_\delta+4\delta+4\\
		&\leq (x|y)_p+ M_\delta+4\delta+6.
	\end{aligned}
	\eeq
	On the other hand, 
	\beqq
	\begin{aligned}
		(x|y)_p
		&=\ell_d(\alpha_1[p_{\alpha_1},p])\leq |p_{\alpha_1}-p|+2\\
		&\leq |p-x_{\alpha}|+|x_{\alpha}-x_{\alpha_1}|+|x_{\alpha_1}-p_{\alpha_1}|+2\\
		&\leq |p-x_0|+M_\delta+4\delta+8.
	\end{aligned}
	\eeqq
	Thus
	\begin{equation*}
		\bigl||p-x_0|-(x|y)_p\bigr|\leq M_\delta+4\delta+8.
	\end{equation*}
	In particular,
	\begin{equation}\label{eq:density-center-qg}
		\rho_\varepsilon(x_0)\asymp_{\delta} e^{-\varepsilon(x|y)_p}.
	\end{equation}
	
	We now estimate $d_\varepsilon(x,y)$. First suppose that $\varepsilon|x-y|\leq 1$. By \eqref{eq:prop of curve}, for every $u\in\gamma$, 
	\[
	\varepsilon|x_0-u|\leq \varepsilon\ell_d(\gamma)\leq 2\varepsilon|x-y|\leq 2.
	\]
	Hence by \eqref{eq:inequality}, for $u\in\gamma$, we have
	\beq\label{5.4}
	e^{-2}\rho_\varepsilon(x_0)\leq \rho_\varepsilon(u)\leq e^{2}\rho_\varepsilon(x_0).
	\eeq
	Since $|x-y|\leq\ell_d(\gamma)\leq 2|x-y|$ by \eqref{eq:prop of curve}, and together with \eqref{5.4}, it follows that
	\[
	d_\varepsilon(x,y)\leq \int_{\gamma}\rho_\varepsilon(u)\,ds\leq e^{2}\rho_\varepsilon(x_0)\ell_d(\gamma)\leq 2e^{2}\rho_\varepsilon(x_0)|x-y|.
	\]
	Moreover, applying Theorem~\ref{thm:G-H} and \eqref{5.4},
	\[
	d_\varepsilon(x,y)\geq \frac{1}{K}\ell_\varepsilon(\gamma)=\frac{1}{K}\int_{\gamma}\rho_\varepsilon(u)\,ds
	\geq \frac{e^{-2}}{K}\rho_\varepsilon(x_0)|x-y|,
	\]
	where $K=K(\delta)$.
	
	Assume next that $\varepsilon|x-y|>1$. Using \cite[Lemma 4.1]{Allu} with $z=x_\alpha$ on $\alpha$, we obtain for any $u_0\in\alpha[x,x_\alpha]$, 
	\[
	|p-u_0|\geq |p-x_\alpha|+|u_0-x_\alpha|-8\delta-16,	
	\]
	By symmetry, for any $u_0\in\alpha[x_\alpha, y]$, 
	\[
	|p-u_0|\geq |p-x_\alpha|+|u_0-x_\alpha|-8\delta-16,	
	\]
	Therefore, for any $u_0\in\alpha$,
	\[
	|p-u_0|\geq |p-x_\alpha|+|u_0-x_\alpha|-8\delta-16,	
	\]
	Then for any $u\in\gamma$, there exists $u_0'\in\alpha$ such that $|u-u_0'|\leq M_\delta$ by Lemma~\ref{thm:stability} and using \eqref{5.1}, we obtain
	\beqq
	\begin{aligned}
		|p-u|&\geq|p-u_0'|-|u-u_0'|\geq|p-x_\alpha|+|u_0'-x_\alpha|-M_\delta-8\delta-16\\
		&\geq|p-x_0|-|x_0-x_\alpha|+|u-x_0|-|u_0'-u|-|x_0-x_\alpha|-M_\delta-8\delta-16\\
		&\geq|p-x_0|+|u-x_0|-\kappa_1,
	\end{aligned}
	\eeqq
	where $\kappa_1=4M_\delta+8\delta+16$.
	
	Parametrize the two components of $\gamma\setminus\{x_0\}$ by arclength from $x_0$. Since $\gamma \in \Lambda_{x,y}^{1,1}(X,d)$, a point at arclength distance $t$ from $x_0$ satisfies $|u-x_0|\geq t-1$. Then we have
	\beqq
	\begin{aligned}
		d_\varepsilon(x,y)
		&\leq \int_{\gamma}\rho_\varepsilon(u)\,ds
		\leq 2e^{\varepsilon \kappa_1}\rho_\varepsilon(x_0)\int_0^\infty e^{-\varepsilon (t-1)}\,dt
		\leq \frac{2e^{\varepsilon_0 (\kappa_1+1)}}{\varepsilon}\rho_\varepsilon(x_0).
	\end{aligned}
	\eeqq

	For the reverse inequality, one of the two components of $\gamma\setminus\{x_0\}$ has length at least $\frac12\ell_d(\gamma)\geq \frac12|x-y|>\frac1{2\varepsilon}.$
	Hence by Theorem~\ref{thm:G-H}, we obtain
	\beqq
	\begin{aligned}
		d_\varepsilon(x,y)
		\geq \frac{1}{K}\ell_\varepsilon(\gamma)=\frac{1}{K}\int_{\gamma}\rho_\varepsilon(u)\,ds
		\geq \frac{1}{K}\rho_\varepsilon(x_0)\int_0^{1/(2\varepsilon)}e^{-\varepsilon t}\,dt
		=\frac{1-e^{-1/2}}{K\varepsilon}\rho_\varepsilon(x_0),
	\end{aligned}
	\eeqq
where $K=K(\delta)$.
	
Combining the two cases with \eqref{eq:density-center-qg}, we obtain
	\[
	d_\varepsilon(x,y)\asymp_{\delta}
	\frac{e^{-\varepsilon(x|y)_p}}{\varepsilon}
	1\wedge\varepsilon|x-y|,
	\]
	which proves \eqref{eq:twopoint}.
\end{proof}
\begin{proof}[Proof of Proposition~\textup{\ref{prop:boundary-id}}]
That	$\Phi$ is well-defined, injective and quasiisometry follows from 
	\cite[Proposition~4.13]{BHK} or \cite[Theorem~1.3]{AlluJose}.
	It remains to prove that $\Phi$ is surjective.  
	
	Let $\xi\in\partial X_\varepsilon$. Then there exists a sequence $\{x_n\}\subset X$ 
	such that $x_n\stackrel{d_\varepsilon}{\longrightarrow}\xi$. Hence
	$d_\varepsilon(x_n,x_m)\to 0$ as $n,m\to\infty$, 
	and \eqref{eq:twopoint} implies
	$$\frac{e^{-\varepsilon(x_n|x_m)_p}}{\varepsilon} 1\wedge\varepsilon|x_n-x_m|\to 0.$$
	If $|x_n-x_m|\to 0$ as $n,m\to\infty$, then by completeness of $X$, there exists $x\in X$ with $x_n\stackrel{d}{\longrightarrow}x$.  
	By \eqref{eq:twopoint}, we would have $x_n\stackrel{d_\varepsilon}{\longrightarrow}x$,  which contradicts the fact that $\xi\in\partial X_\varepsilon$ whereas $x\in X$ is an interior point. Therefore
	$|x_n-x_m|\not\to 0$ as $n,m\to\infty$.
	Consequently,
	$(x_n|x_m)_p\xrightarrow{n,m\to\infty}\infty,$
	and the sequence $\{x_n\}$ defines a point of $\partial_G X$. The surjectivity of $\Phi$ follows.	
\end{proof}



\end{document}